\documentclass[english,reqno]{amsart}

\usepackage{amsmath, appendix, ulem}
\usepackage[nobysame]{amsrefs}
\usepackage{amssymb, color}
\usepackage[margin=1in]{geometry}

\usepackage{mathrsfs}
\usepackage{graphicx}
\usepackage{subfig}
\usepackage{float}
\usepackage{epsf}
\usepackage{hyperref}
\usepackage{titletoc}

\graphicspath{{../Figures/}}

\numberwithin{equation}{section}

\newtheorem{lem}{Lemma}[section]
\newtheorem{thm}{Theorem}[section]

\theoremstyle{remark}
\newtheorem{rmk}{Remark}[section]

\renewcommand{\tilde}{\widetilde}
\renewcommand{\hat}{\widehat}

\newcommand{\nn}{\nonumber}

\newcommand{\del}{\partial}
\newcommand{\Denote}{\stackrel{\Delta}{=}}

\newcommand{\ds}{\, {\rm d} s}

\newcommand{\dtheta}{\, {\rm d} \theta}
\newcommand{\dtau}{\, {\rm d} \tau}
\newcommand{\deta}{\, {\rm d} \eta}
\newcommand{\Disp}{\displaystyle}

\newcommand{\Ss}{{\mathbb{S}}}

\newcommand{\Eps}{\epsilon}
\newcommand{\Ni}{\noindent}

\newcommand{\BigO}{{\mathcal{O}}}
\newcommand{\Sgn}{\text{sgn}}

\newcommand{\domega}{\, {\rm d}\omega}

\newcommand{\NN}{\mathbb{N}}

\newcommand{\rd}{\mathrm{d}}

\newcommand{\abs}[1]{\left\lvert#1\right\rvert}
\newcommand{\norm}[1]{\left\lVert#1 \, \right\rVert}
\newcommand{\vint}[1]{\left\langle#1\right\rangle}

\newcommand{\vpran}[1]{\left(#1\right)}

\begin{document}

\title{Validity and Regularization of Classical Half-Space Equations}

\author{Qin Li} 
\address{Mathematics Department, University of Wisconsin-Madison, 480 Lincoln Dr., Madison, WI 53705 USA.}
\email{qinli@math.wisc.edu}
\author{Jianfeng Lu}
\address{Departments of Mathematics, Physics, and Chemistry, Duke University, Box 90320, Durham, NC 27708 USA.}
\email{jianfeng@math.duke.edu}
\author{Weiran Sun}
\address{Department of Mathematics, Simon Fraser University, 8888 University Dr., Burnaby, BC V5A 1S6, Canada}
\email{weirans@sfu.ca}

%
\thanks{We would like to express our gratitude to the support from the
  NSF research network grant RNMS11-07444 (KI-Net). The research of Q.L.~was supported in part by the National Science Foundation under award DMS-1318377 and DMS-1619778. The research of J.L.~was supported in part by the National Science Foundation under award DMS-1454939. The research of W.S.~was supported in part by the NSERC Discovery Individual Grant \#611626.}
\begin{abstract}
Recent result \cite{WG2014} has shown that over the 2D unit disk, the classical half-space equation (CHS) for the neutron transport does not capture the correct boundary layer behaviour as long believed. In this paper we develop a regularization technique for CHS to any arbitrary order and use its first-order regularization to show that  in the case of the 2D unit disk, although CHS misrepresents the boundary layer behaviour, it does give the correct boundary condition for the interior macroscopic (Laplace) equation. Therefore CHS is still a valid equation to recover the correct boundary condition for the interior Laplace equation over the 2D unit disk.
\end{abstract}
%
\maketitle

\section{Introduction}
It is well-known that for kinetic equations with a small Knudsen number  imposed on a bounded domain, a thin layer (coined the Knudsen layer) will form near the boundary of the domain. The kinetic density distribution function changes sharply within this layer from the given arbitrary kinetic boundary conditions to more restrictive interior states, such as those near the equilibrium states.  To make use of the particular structure of the interior state and reduce the computational cost of solving the full scaled kinetic equation over the whole domain, one classical way is to introduce a half-space equation to capture the boundary layer behaviour. In particular, the end-states of the half-space equation will serve as the boundary conditions for the interior equation. 

In this paper, we consider the scaled steady-state isotropic neutron transport equation
\begin{equation}
\begin{gathered} \label{eq:full-kinetic}
    \omega \cdot \nabla_x F + \frac{1}{\Eps} \vpran{F - \vint{F}} = 0 \,,
\qquad
    (x, \omega) \in D \times \Ss^1 \,,
\\
    F \Big|_{\del D} = h(x, \omega) \,,
\qquad
    \omega \cdot n < 0 \,,
\end{gathered}
\end{equation}
where $F = F(x, \omega)$  is the density function and $x, \omega$ are the spatial and velocity variables respectively. The spatial domain $D$ is the unit disk with outward normal $n$. The speed of the particles is constant and is scaled to one so that $\omega \in \Ss^1$. We also have $\vint{F} = \frac{1}{2\pi}\int_{\Ss^1} F(x, \omega) \domega$.

The classical half-space equation associated with~\eqref{eq:full-kinetic} can be derived through asymptotic analysis. Since one of our main objectives is to compare the classical half-space equation to an $\Eps$-Milne equation constructed in~\cite{WG2014}, we adopt similar notations as in~\cite{WG2014}. In particular, we use the polar coordinates within the boundary layer together with the stretched spatial variable such that 
\begin{align*}
   x = (r \cos\phi, r \sin\phi) \,,
\qquad
   \omega = (- \sin\xi, -\cos\xi) \,,
\qquad
  \eta = \frac{1-r}{\Eps} \,,
\qquad
  \theta = \phi + \xi \,.
\end{align*}
In these notations, the leading-order classical half-space equation has the form
\begin{align}
   \sin\theta \frac{\del f_0}{\del \eta} + f_0& - \vint{f_0} = 0 \,,  \label{eq:kinetic-classical-0}
\\
   f_0 \big|_{\eta=0} &= h_0(\theta, \phi) \,,  \label{eq:kinetic-classical-1-0}
\qquad
   \sin\theta > 0 \,,
\\
   f_0 \to  &\, f_{0,\infty}   
\qquad \hspace{0.3cm}
   \text{\qquad as $\eta \to \infty$,}  \label{eq:kinetic-classical-2-0}
\end{align}
where $\vint{f_0} = \frac{1}{2\pi} \int_{-\pi}^{\pi} f_0(\eta, \theta, \phi) \dtheta$ and     
$h_0(\theta, \phi) = h(x ,\omega)$ for $x \in \del D$. Meanwhile, the leading-order interior solution satisfies
\begin{align}
    - \Delta_x U_0(x) &= 0 \,,
\qquad
    x \in D \,,    \label{eq:U-0-interior}
\\
   U_0 \Big|_{\del D} &= u_0 (x) \,,   \label{bdry:U-0}
\end{align}
where $u_0 (x) = f_{0, \infty}(\phi)$ with $x = (\cos\phi, \sin\phi)$.

The question of finding the leading-order approximate solution to~\eqref{eq:full-kinetic} has been considered as settled since the work \cite{BLP:79}, in which it was shown that
\begin{align} \label{error:classical}
    \norm{F - (f_0 - f_{0, \infty} + U_0)}_{L^\infty(D \times \Ss^1)}
= \BigO(\Eps) \,.
\end{align}
However, in a series of recent works \cite{WG2014, WYG2016, GW2016} the authors constructed counterexamples  such that 
\begin{align} \label{error:O-1-L-inf} 
    \norm{F - (f_0 - f_{0, \infty} + U_0)}_{L^\infty(D \times \Ss^1)}
= \BigO(1) \,.
\end{align}
This indicates that the classical half-space equation fails to capture the correct boundary layer behaviour. In \cite{WG2014} where the unit disk is considered, the authors introduced an $\Eps$-Milne equation which has the form
\begin{align} 
   \sin\theta \frac{\del f_{0,\Eps}}{\del \eta} 
   - \frac{\Eps \psi(\Eps \eta)}{1- \Eps \eta} &\cos\theta \frac{\del f_{0,\Eps}}{\del \theta}
   + f_{0,\Eps} - \vint{f_{0,\Eps}} = 0 \,, \label{eq:kinetic-eps-0}
\\
   &f_{0,\Eps} \big|_{\eta=0} = h_0(\theta, \phi) \,,
\qquad
   \sin\theta > 0 \,,  \label{eq:kinetic-eps-1-0}
\\
   &f_{0,\Eps} \to  f_{\infty,\Eps}
\qquad \hspace{0.2cm}
   \text{\qquad as $\eta \to \infty$,} \label{eq:kinetic-eps-2-0}
\end{align}
where $\psi$ is a proper cutoff function. Using this new system as the boundary layer equation, they have proved that 
\begin{align}
    \norm{F - (f_{0,\Eps} - f_{\infty, \Eps} + U_{0,\Eps})}_{L^\infty(D \times \Ss^1)}
= \BigO(\Eps) \,,
\end{align}
where $U_{0,\Eps}$ satisfies the Laplace equation on the disk with the boundary condition given by $f_{\infty, \Eps}$. Later this result is generalized to the annulus \cite{WYG2016} and the general 2D convex domains with diffusive boundary conditions \cite{GW2016}. Similar $\Eps$-Milne equations are used in \cite{WYG2016, GW2016} as the boundary layer equations. 

These surprising results show that the $\Eps$-Milne systems are indeed the correct boundary layer equations. The seemingly small $\Eps$-term in~\eqref{eq:kinetic-eps-0} plays a major role which makes the equation singular. This then suggests challenges on numerical computations to find the proper boundary conditions for the interior equation, since directly solving the $\Eps$-Milne to obtain the end-states as the correct boundary conditions is probably as expensive as solving the original full scaled kinetic equation~\eqref{eq:full-kinetic}. In this sense, despite its obvious theoretical importance, the $\Eps$-Milne equation does not seem to serve the original purpose of reducing computational costs. 

In this paper, we address the validity of the classical half-space equation by using the $\Eps$-Milne system as an intermediate equation. Our first main result is: although~\eqref{error:classical} does not hold on the entire disk, it turns out that the away from the boundary layer, the interior solution $U_0$ generated from the end-state $f_{0, \infty}$ of the classical half-space equation still gives a correct leading-order approximation. More precisely, there exists a constant $C(\alpha)$ such that
\begin{align} \label{error:new}
    \norm{F - U_0}_{L^\infty(\alpha D \times \Ss^1)} \leq C(\alpha) \Eps^{2/3} \,,
\end{align}
where $\alpha D = \{(\alpha x, \alpha y) | (x, y) \in D\}$ for any $0 < \alpha < 1$. 
Therefore, the $\BigO(1)$-error of the approximate solution $f_0 + U_0$ is restricted to the thin boundary layer and does not propagate inside. Note that $\Eps^{2/3}$ may not be the optimal decay rate. 

One of the main tools that we develop to prove~\eqref{error:new} is a regularization procedure designed particularly for the classical half-space equation. It is known ( see for example \cite{CZ1967, TF2013, Chen2013, CFLT2016, Wing1962} and references therein) that regardless of the regularity of the given incoming data, the half-space equation~\eqref{eq:kinetic-classical}-\eqref{eq:kinetic-classical-2} has a generic jump at $\theta = 0$ as well as a logarithmic singularity as $\theta \to 0^-$. In fact it is exactly this singularity that renders the failure of the classical  error estimate~\eqref{error:classical}. For the purpose of proving~\eqref{error:new}, we show a first-order regularization that makes the modified solution Lipschitz. In the second part of this paper, we generalize this procedure to obtain regularizations of solutions to the classical half-space equation to any arbitrary order. The higher-order regularization will be useful for comparing the classical half-space equation with the $\Eps$-Milne system over general domains. We leave the general geometry to later work to avoid overburdening the current paper.

The rest of the paper is laid out as follows. In Section~\ref{Sec:error-estimate} we use the regularization technique and the $\Eps$-Milne equation to show~\eqref{error:new}. In Section~\ref{Sec:numerics}, we show numerical evidence of the non-convergence of the classical approximation in the $L^\infty$-norm and convergence in the $L^2$-norm. We also numerically compare the classical half-space equation with the $\Eps$-Milne equation. In Section~\ref{Sec:regularization}, we show the general regularization of the half-space equation to arbitrary orders. 

\section{Comparison with Wu-Guo's $\Eps$-Milne equation} \label{Sec:error-estimate}
In this section we compare the end-states of the classical half-space equation and the $\Eps$-Milne equation. To simplify the notation, we will use $f$ and $f_\Eps$ in place of $f_0$ and $f_{0, \Eps}$ for their solutions. The two equations are repeated below: the classical half-space equation
\begin{align}
   \sin\theta \frac{\del f}{\del \eta} + f& - \vint{f} = 0 \,,  \label{eq:kinetic-classical}
\\
   f \big|_{\eta=0} &= h_0(\theta, \phi) \,,  \label{eq:kinetic-classical-1}
\qquad
   \sin\theta > 0 \,,
\\
   f \to  &\, f_\infty(\phi)   
\qquad \hspace{0.2cm}
   \text{\qquad as $\eta \to \infty$,}  \label{eq:kinetic-classical-2}
\end{align}
and the $\Eps$-Milne equation
\begin{align} 
   \sin\theta \frac{\del f_\Eps}{\del \eta} 
   - \frac{\Eps \psi(\Eps \eta)}{1- \Eps \eta} &\cos\theta \frac{\del f_\Eps}{\del \theta}
   + f_\Eps - \vint{f_\Eps} = 0 \,, \label{eq:kinetic-eps}
\\
   &f_\Eps \big|_{\eta=0} = h_0(\theta, \phi) \,,
\qquad 
   \sin\theta > 0 \,,  \label{eq:kinetic-eps-1}
\\
   &f_\Eps \to  f_{\infty,\Eps}(\phi)
\qquad \hspace{0.2cm}
   \text{\qquad as $\eta \to \infty$,} \label{eq:kinetic-eps-2}
\end{align}
where $f_\infty, f_{\infty, \Eps}$ only depend on $\phi$ and $\psi$ is a smooth cut-off function such that 
\begin{align*}
    \psi(r) 
    = \begin{cases}
        1 \,, & 0 \leq r \leq 1/2 \,, \\[2pt]
        0 \,, & r \geq 3/4 \,.
       \end{cases}
\end{align*}
Since the angular variable $\phi$ does not play a role in our analysis, we will suppress it in the notations from now on unless otherwise specified. 

The well-posedness of~\eqref{eq:kinetic-classical}-\eqref{eq:kinetic-classical-2} and~\eqref{eq:kinetic-eps}-\eqref{eq:kinetic-eps-2} are thoroughly studied in \cite{CoronGolseSulem:88} and \cite{WG2014}. In \cite{WG2014}, it is pointed out that the second term on the left involving $\Eps$ in equation~\eqref{eq:kinetic-eps} has a non-trivial effect which induces an order $\BigO(1)$
difference between $f$ and $f_\Eps$ measured in the $L^\infty$-norm over the whole domain. However, it is not clear from the analysis in \cite{WG2014} whether the end-states $f_\infty$ and $f_{\infty, \Eps}$ will differ by order $\BigO(1)$ as well. In this section we show that in fact they only differ on a scale which vanishes with $\Eps$. The main result is 
\begin{thm} \label{thm:main}
Let $f, f_\Eps$ be the solutions to equations~\eqref{eq:kinetic-classical}-\eqref{eq:kinetic-classical-2} and~\eqref{eq:kinetic-eps}-\eqref{eq:kinetic-eps-2} respectively. Then there exists a constant $C_0$ independent of $\Eps$ such that
\begin{align*}
   \abs{f_\infty - f_{\infty, \Eps}} \leq C_0 \Eps^{2/3} \,.
\end{align*}
\end{thm}

\begin{rmk}
The convergence rate $\Eps^{2/3}$ may not be optimal.
\end{rmk}

\Ni {\it Notation.} In this paper we use $\Xi_1 = \BigO(\alpha)$ to denote the upper bound
\begin{align*}
    \abs{\Xi_1} \leq C_1 \alpha \,,
\end{align*}
where $C_1$ is independent of $\alpha$. This is different from the somewhat conventional notation that $\Xi_1 = \BigO(\alpha)$ means $\Xi_1$ is comparable with $\alpha$ in the way that it is bounded both from above and below by an order of $\alpha$. 
We also use $\Xi_2 = o(\alpha)$ to denote that 
\begin{align*}
    \frac{\abs{\Xi_2}}{\alpha} \to 0 
\quad
   \text{as $\alpha \to 0$.}
\end{align*}

\medskip

Before proving the main result, we first state a few lemmas. The first one shows an ``almost" conservation law for the $\Eps$-Milne system~\eqref{eq:kinetic-eps}-\eqref{eq:kinetic-eps-2}.

\begin{lem} \label{lem:conservation}
Let $f_\Eps$ be the solution to~\eqref{eq:kinetic-eps}-\eqref{eq:kinetic-eps-2}. Then 
we have
\begin{align*}
     \int_{-\pi}^\pi f_\Eps(0, \theta) \sin^2\theta \dtheta 
     = \int_{-\pi}^\pi f_{\infty, \Eps} \sin^2\theta \dtheta  + \BigO(\Eps)
     = \pi f_{\infty, \Eps} + \BigO(\Eps) \,.
\end{align*}
\end{lem}
\begin{proof}

Recall that $f_\Eps$ satisfies the conservation property \cite{WG2014}: 
\begin{align*}
    \vint{f_\Eps \sin\theta} 
    = \frac{1}{2\pi}\int_{-\pi}^\pi f_\Eps(\eta, \theta) \sin\theta \dtheta = 0 
\qquad
    \text{for any $\eta \geq 0$.}
\end{align*}
Multiplying $\sin\theta$ to~\eqref{eq:kinetic-eps} and integrating in $\theta$ gives
\begin{align} \label{eq:f-sin-square}
   \frac{\rd}{\rd \eta} \int_{-\pi}^\pi f_\Eps \sin^2\theta \dtheta
   - \frac{1}{2} \frac{\Eps \psi(\Eps \eta)}{1- \Eps \eta} 
     \int_{-\pi}^\pi \sin(2\theta) \frac{\del f_\Eps}{\del \theta} \dtheta = 0 \,.
\end{align}
The term $\int_{-\pi}^\pi \sin(2\theta) \frac{\del f_\Eps}{\del \theta} \dtheta$ can be re-written as
\begin{align*}
   \int_{-\pi}^\pi \sin(2\theta) \frac{\del f_\Eps}{\del \theta} \dtheta
& = -2 \int_{-\pi}^\pi f_\Eps(\eta, \theta) \cos(2\theta) \dtheta
 = -2 \int_{-\pi}^\pi f_\Eps(\eta, \theta) \vpran{1 - 2 \sin^2\theta} \dtheta
\\
& = -2 \int_{-\pi}^\pi f_\Eps \dtheta
       + 4 \int_{-\pi}^\pi f_\Eps \sin^2\theta \dtheta \,.
\end{align*}
Therefore \eqref{eq:f-sin-square} becomes
\begin{align} \label{eq:f-sin-square-1}
   \frac{\rd}{\rd \eta} \int_{-\pi}^\pi f_\Eps \sin^2\theta \dtheta
   - \frac{2\Eps \psi(\Eps \eta)}{1- \Eps \eta} 
     \int_{-\pi}^\pi f_\Eps \sin^2\theta \dtheta 
   =  - \frac{\Eps \psi(\Eps \eta)}{1- \Eps \eta} 
       \int_{-\pi}^\pi f_\Eps \dtheta \,.
\end{align}
Using the notation in \cite{WG2014}, we denote 
\begin{align} \label{def:V-Eps}
   \del_\eta V_\Eps(\eta) = \frac{\Eps \psi(\Eps \eta)}{1- \Eps \eta} \,,
\qquad
   V_\Eps(0) = 0 \,.
\end{align}
Then $0 \leq V_\Eps \leq V_\infty < \infty$ and 
\begin{align} \label{eq:V-infty}
   V_\infty = \int_0^\infty \frac{\Eps \psi(\Eps \eta)}{1- \Eps \eta}  \deta
   = \int_0^\infty \frac{\psi(\eta)}{1-\eta}  \deta
\qquad
   \text{is bounded and independent of $\Eps$.}
\end{align}
We further introduce the notation
\begin{align*}
    H(\eta) = \frac{1}{2\pi} \int_{-\pi}^\pi f_\Eps(\eta, \theta) \sin^2\theta \dtheta \,.
\end{align*}
Then \eqref{eq:f-sin-square-1} becomes
\begin{align*}
    \frac{\rd}{\rd \eta} H(\eta) - 2 \vpran{\del_\eta V_\Eps} H(\eta)
    = \vpran{-\del_\eta V_\Eps} \vint{f_\Eps} \,.
\end{align*}
Therefore, 
\begin{align*}
    \frac{\rd}{\rd \eta} \vpran{e^{-2V_\Eps} H(\eta)}
    = e^{-2V_\Eps} \vpran{-\del_\eta V_\Eps} \vint{f_\Eps} \,.
\end{align*}
Integrating from $\eta = 0$ to $\eta = \infty$ gives
\begin{align} \label{H-infty}
    H_\infty 
    = e^{2V_\infty} 
        \vpran{H_0 - \int_0^\infty e^{-2V_\Eps} \del_\tau V_\Eps \vint{f_\Eps}(\tau) \dtau} \,,
\end{align}
where
\begin{align*}
    H_\infty = \frac{1}{2\pi} \int_{-\pi}^\pi f_{\infty, \Eps} \sin^2\theta \dtheta
    = \frac{1}{2} f_{\infty, \Eps} \,,
\qquad
    H_0 = \frac{1}{2\pi} \int_{-\pi}^\pi f_{\Eps}(0, \theta) \sin^2\theta \dtheta \,.
\end{align*}
By \eqref{eq:V-infty} the constant $V_\infty$ only depends on the choice of the cutoff function $\psi$. 
The integral term can be reformulated as
\begin{align*}
   \int_0^\infty e^{-2V_\Eps} \del_\tau V_\Eps \vint{f_\Eps}(\tau) \dtau
&= f_{\infty, \Eps} \int_0^\infty e^{-2V_\Eps} \del_\tau V_\Eps \dtau
   + \int_0^\infty e^{-2V_\Eps} \del_\tau V_\Eps \vint{f_\Eps - f_{\infty, \Eps}}(\tau) \dtau
\\
&= \frac{1}{2} f_{\infty, \Eps} \vpran{1 - e^{-2 V_\infty}}
   + \int_0^\infty e^{-2V_\Eps} \del_\tau V_\Eps \vint{f_\Eps - f_{\infty, \Eps}}(\tau) \,.
\end{align*}
Therefore, the right-hand side of~\eqref{H-infty} satisfies
\begin{align*}
& \quad \,
   e^{2V_\infty} \vpran{H_0 - \int_0^\infty e^{-2V_\Eps} \del_\tau V_\Eps \vint{f_\Eps}(\tau) \dtau}
\\
&= e^{2V_\infty} \vpran{H_0 - \frac{1}{2} f_{\infty, \Eps}}
   + \frac{1}{2} f_{\infty, \Eps}
   - e^{2V_\infty} \int_0^\infty e^{-2V_\Eps} \del_\tau V_\Eps \vint{f_\Eps - f_{\infty, \Eps}}(\tau) \,.
\end{align*}
By \eqref{H-infty} and $H_\infty = \frac{1}{2} f_{\infty, \Eps}$, we then have 
\begin{align*}
    H_0 - \frac{1}{2} f_{\infty, \Eps}
    = \int_0^\infty e^{-2V_\Eps} \del_\tau V_\Eps \vint{f_\Eps - f_{\infty, \Eps}}(\tau) \,.
\end{align*}
Denote
\begin{align*}
   G = \int_0^\infty e^{-2V_\Eps} \del_\tau V_\Eps \vint{f_\Eps - f_{\infty, \Eps}}(\tau) \,.
\end{align*}
It has been shown in \cite{WG2014} that there exists $\kappa_0 > 0$ such that 
\begin{align*}
     \abs{f_\Eps - f_{\infty, \Eps}} \leq C e^{-\kappa_0 \eta} \,,
\end{align*}
where $C$ is independent of $\Eps$.  Let $d_0 > 0$ be a constant such that 
\begin{align*}
   1 - \Eps \eta \geq d_0 > 0 
\qquad
   \text{on $\mathrm{supp}\, \psi(\Eps \eta)$.}
\end{align*}
We have
\begin{align*}
   |G|
\leq \norm{\psi}_{L^\infty} \frac{\Eps}{d_0} \int_0^\infty e^{-\alpha_1 \eta} \deta
\leq C \Eps \,,
\end{align*}
where $C$ is independent of $\Eps$. Thus,
\begin{align*}
     \abs{H_0 - \frac{1}{2} f_{\infty, \Eps}} = \abs{G} \leq C \Eps \,.
\end{align*}
This is equivalent to 
\begin{align*}
    \int_{-\pi}^\pi f_\Eps(0, \theta) \sin^2\theta \dtheta
 = \pi f_{\infty, \Eps} + \BigO(\Eps)
 = \int_{-\pi}^\pi f_{\infty, \Eps} \sin^2\theta \dtheta + \BigO(\Eps) \,,
\end{align*}
which proves the desired ``almost" conservation property.
\end{proof}

Next we show a stability result for both the classical half-space and the $\Eps$-Milne equation. 
\begin{lem} \label{lem:stability}
Let $f, f_\Eps$ be the solutions to equations~\eqref{eq:kinetic-classical}-\eqref{eq:kinetic-classical-2} and~\eqref{eq:kinetic-eps}-\eqref{eq:kinetic-eps-2} respectively. Then there exists a constant $C$ independent of $\Eps$ such that
\begin{align*}
   \abs{f_{\infty, \Eps}}
\leq 
   \frac{2}{\pi}\int_0^\pi \abs{f_\Eps(0, \theta)} \sin\theta \dtheta + C \Eps \,,
\qquad
   \Eps > 0 \,,
\end{align*}
and for the classical half-space equation it holds that
\begin{align*}
   \abs{f_{\infty}}
\leq 
   \frac{2}{\pi} \int_0^\pi \abs{f(0, \theta)} \sin\theta \dtheta \,.
\end{align*}

\end{lem}
\begin{proof}
Multiply $\Sgn(f_\Eps)$ to equation~\eqref{eq:kinetic-eps} and integrate in $\theta$. Then we have 
\begin{align*}
    \frac{\rd}{\rd \eta} \int_{-\pi}^\pi \sin\theta \abs{f_\Eps} \dtheta
    - \vpran{\del_\eta V_\Eps} \int_{-\pi}^\pi \sin\theta \abs{f_\Eps} \dtheta
\leq 0 \,.
\end{align*}
Therefore, 
\begin{align*}
   \frac{\rd}{\rd \eta} \vpran{e^{-V_\Eps} \int_{-\pi}^\pi \sin\theta \abs{f_\Eps} \dtheta} \leq 0 \,,
\end{align*}
which gives
\begin{align*} 
    \int_{-\pi}^\pi \sin\theta \abs{f_\Eps(0, \theta)} \dtheta \geq 0 \,.
\end{align*}
Thus we have
\begin{align} \label{bound:f-neg-L-1}
   0 \leq \int_{-\pi}^0 \abs{\sin\theta} \abs{f_\Eps(0, \theta)} \dtheta
   \leq \int_0^\pi \sin\theta \abs{f_\Eps(0, \theta)} \dtheta \,.
\end{align}
This in particular shows that
\begin{align*}
   \abs{\int_{-\pi}^\pi  f_\Eps (0, \theta) \sin^2\theta \dtheta}
\leq
  \int_{-\pi}^\pi  \abs{f_\Eps (0, \theta)} \abs{\sin\theta} \dtheta
\leq
  2\int_0^\pi \abs{f_\Eps (0, \theta)} \sin\theta \dtheta \,.
\end{align*}
By the ``almost" conservation law in Lemma~\ref{lem:conservation}, we have
\begin{align*}
   \abs{f_{\infty, \Eps}} 
 = \frac{1}{\pi}\abs{\int_{-\pi}^\pi f_\Eps(0, \theta) \sin^2\theta \dtheta} + \BigO(\Eps)
 \leq 
  \frac{2}{\pi} \int_0^\pi \abs{f_\Eps (0, \theta)} \sin\theta \dtheta + \BigO(\Eps) \,.
\end{align*}
For the classical half-space equation, the estimates are similar and we only need to remove the error term $\BigO(\Eps)$ since the strict conservation holds. 
\end{proof}

The main reason that $f_\Eps$ has a finite difference from $f$ near the boundary is because $f$ has insufficient regularity in terms of $\eta$ and $\theta$. Specifically, it has been shown \cite{WG2014} that in general $\frac{\del f}{\del \eta}$ is not uniformly bounded. In the following lemma, we show that by slightly changing the incoming data, we can find a solution $\tilde f$ to the classical half-space equation such that $\tilde f \in W^{1, \infty}(\deta\dtheta)$.
Without loss of generality, we assume that the incoming data $h_0$ is not a constant function and
\begin{align} \label{assump:h-0}
    0 \leq h_0 \leq 1 \,,
\qquad
    \theta \in (0, \pi) \,.
\end{align}

\begin{lem} \label{lem:modification}
Let $f$ be the solution to~\eqref{eq:kinetic-classical}-~\eqref{eq:kinetic-classical-2}. Then for any $0 < \alpha_\Eps < \pi/4$, 
\medskip

\Ni (a) the classical half-space equation~\eqref{eq:kinetic-classical} has a solution $\tilde f \in W^{1, \infty}(\deta\dtheta)$ with a modified incoming data $\tilde h_0(\theta)$ such that 
\begin{align} \label{ineq:infty-difference}
   \int_{-\pi}^\pi \abs{\tilde f(0, \theta) - f(0, \theta)} \abs{\sin\theta} \dtheta
\leq \alpha_\Eps^2 \,,
\end{align}
\Ni (b) There exist two constants $C > 0$ and $0 < \kappa_0 < 1$ independent of $\Eps$ such that we have the bounds
\begin{align} \label{bound:tildef-eta}
   \norm{e^{\kappa_0 \eta} \frac{\del \tilde f}{\del \eta}}_{L^\infty(\deta\dtheta)}
  +    \norm{e^{\kappa_0 \eta} \frac{\del \tilde f}{\del \theta}}_{L^\infty(\deta\dtheta)}
\leq 
   \frac{C}{\alpha_\Eps} \,.
\end{align}
 \end{lem}
\begin{proof}
(a) We will slightly change the incoming data for $f$ near $\theta = 0$ to obtain the desired $\tilde f$.
Note that if $\tilde f$ solves the classical half-space equation~\eqref{eq:kinetic-classical}, then $\frac{\del \tilde f}{\del \eta}$ satisfies
\begin{align} 
    \sin\theta \frac{\del}{\del \eta} \vpran{\frac{\del \tilde f}{\del \eta}}
    &+ \frac{\del \tilde f}{\del \eta}
    - \vint{\frac{\del \tilde f}{\del \eta}} = 0 \,,  \nn
\\
   \frac{\del \tilde f}{\del \eta} \Big|_{\eta = 0} 
   = &-\frac{\tilde f(0, \theta) -  \bigl\langle\tilde f\bigr\rangle(0)}{\sin\theta} \,,
\qquad
   \sin\theta > 0 \,. \label{eq:tilde-f}
\\
   \frac{\del \tilde f}{\del \eta} &\to 0
\qquad \hspace{2.5cm}
   \text{as $\eta \to \infty$.} \nn
\end{align}
The end-state of $\frac{\del \tilde f}{\del \eta}$ is zero because for any $\theta \neq 0$, we have
\begin{align*}
    \frac{\del \tilde f}{\del \eta} 
    = \frac{\bigl\langle\tilde f\bigr\rangle - \tilde f}{\sin\theta} 
    \to 0
\qquad
    \text{as $\eta \to \infty$. }
\end{align*}
By the maximum principle for the classical half-space equation, in order to achieve that $\frac{\del \tilde f}{\del \eta} \in L^\infty(\deta\dtheta)$, we only need to make sure that 
\begin{align} \label{bound:tildef-bdry}
    \frac{\tilde f(0, \theta) -  \bigl\langle\tilde f\bigr\rangle(0)}{\sin\theta}
 \in L^\infty(0, \pi) \,.
\end{align}
To this end, we first construct two auxilliary functions. For any given $\alpha_\Eps > 0$, 
let 
\begin{align*}
    \phi_1(\theta) 
 = \begin{cases}
       0, & \text{for $\theta \in [0, \alpha_\Eps] \cup [\pi - \alpha_\Eps, \pi]$}, \\[2pt]
       h_0(\theta), & \text{for $\theta \in [2\alpha_\Eps, \pi - 2\alpha_\Eps]$}, \\[2pt]
       \text{Lipschitz}, & \text{for $\theta \in [0, \pi]$}, 
    \end{cases}
\qquad
    0 \leq \phi_1 \leq 1\,,
\end{align*}
and
\begin{align*}
    \phi_2(\theta) 
 = \begin{cases}
       1, & \text{for $\theta \in [0, \alpha_\Eps] \cup [\pi - \alpha_\Eps, \pi]$}, \\[2pt]
       h_0(\theta), & \text{for $\theta \in [2\alpha_\Eps, \pi - 2\alpha_\Eps]$}, \\[2pt]
       \text{Lipschitz}, & \text{for $\theta \in [0, \pi]$},
    \end{cases}
\qquad
    0 \leq \phi_2 \leq 1\,.
\end{align*}
Let $f_1, f_2$ be the solutions to the half-space equation~\eqref{eq:kinetic-classical} with incoming data $\phi_1, \phi_2$ respectively. Then by the maximum principle again, we have at $\eta = 0$,
\begin{align*}
    f_1(0, 0^+) - \vint{f_1(0, \cdot)} = - \vint{f_1(0, \cdot)} < 0 \,,
\\
    f_2(0, 0^+) - \vint{f_2(0, \cdot)} = 1 - \vint{f_2(0, \cdot)} > 0 \,.
\end{align*}
Therefore, there exists a constant $0 < \lambda_0 < 1$ such that
\begin{align*}
    \lambda_0 \vpran{f_1(0, 0^+) - \vint{f_1(0, \cdot)}}
    + (1 - \lambda_0)
    \vpran{f_2(0, 0^+) - \vint{f_2(0, \cdot)}} = 0 \,.
\end{align*}
Let $\tilde f = \lambda_0 f_1 + (1 - \lambda_0) f_2$. Then $\tilde f$ satisfies
\begin{align}
   \sin\theta \frac{\del \tilde f}{\del \eta} 
   &+ \tilde f - \bigl\langle\tilde f\bigr\rangle = 0 \,,  \label{eq:kinetic-classical-tilde}
\\
   \tilde f \big|_{\eta=0} 
   &= \lambda_0 \phi_1(\theta) + (1-\lambda_0) \phi_2(\theta)
   \Denote \tilde h_0(\theta) \,,  \label{eq:kinetic-classical-1-tilde}
\qquad
   \sin\theta > 0 \,,
\\
   &\tilde f \to  \, \tilde f_\infty   
\qquad \hspace{3.6cm}
   \text{\qquad as $\eta \to \infty$,}  \label{eq:kinetic-classical-2-tilde}
\end{align}
where $\tilde f_\infty$ is constant in $\eta, \theta$. Moreover, $\tilde f$ satisfies that 
\begin{align*}
     \tilde f(0, \theta) \equiv 1 - \lambda_0 = \bigl\langle\tilde f\bigr\rangle(0)
\qquad
     \text{for $\theta \in [0, \alpha_\Eps] \cup [\pi - \alpha_\Eps, \pi]$} \,.
\end{align*}
This in particular shows that 
\begin{align*}
    \frac{\del \tilde f}{\del \eta} \Big|_{\eta=0}
    = -\frac{\tilde f(0, \theta) -  \bigl\langle\tilde f\bigr\rangle(0)}{\sin\theta}    
   = 0 
\qquad
     \text{for $\theta \in [0, \alpha_\Eps] \cup [\pi - \alpha_\Eps, \pi]$} \,.
\end{align*}
By the half-space equation for $\tilde f$ we also have 
\begin{align*}
   \abs{-\frac{\tilde f(0, \theta) -  \bigl\langle\tilde f\bigr\rangle(0)}{\sin\theta}}
\leq 
   \frac{1}{\sin(\alpha_\Eps)}
\qquad
     \text{for $\theta \in [\alpha_\Eps, \pi - \alpha_\Eps]$} \,.
\end{align*}
Thus,
\begin{align} \label{bound:incoming-Tf-del}
   \norm{-\frac{\tilde f(0, \theta) -  \bigl\langle\tilde f\bigr\rangle(0)}{\sin\theta}}_{L^\infty(0, \pi)}
\leq 
   \frac{1}{\sin(\alpha_\Eps)}
\leq 
   \frac{C}{\alpha_\Eps} \,,
\end{align}
where $C$ is independent of $\Eps$.
Applying the maximum principle to~\eqref{eq:tilde-f} then gives
\begin{align*} 
   \norm{\frac{\del \tilde f}{\del \eta}}_{L^\infty(\deta\dtheta)}
\leq 
   \frac{1}{\sin(\alpha_\Eps)} \,.
\end{align*}
In order to show that~\eqref{ineq:infty-difference} holds, we note that by construction,
\begin{align*}
    \tilde f(0, \theta) - f(0, \theta)
  = \begin{cases}
       \lambda_0 \phi_1 + (1-\lambda_0) \phi_2 - h_0(\theta), 
       & \text{for $\theta \in [0, \alpha_\Eps] \cup [\pi - \alpha_\Eps, \pi]$}, \\[2pt]
       0, & \text{for $\theta \in [2\alpha_\Eps, \pi - 2 \alpha_\Eps]$}, \\[2pt]
       \text{Lipschitz} \,, & \text{for $\theta \in [0, \pi]$} \,,
     \end{cases}
\end{align*}
where $\abs{\lambda_0 \phi_1 + (1-\lambda_0) \phi_2 - h_0(\theta)} \leq 1$. 
Thus by~\eqref{bound:f-neg-L-1},
\begin{align*}
      \int_{-\pi}^\pi \abs{\tilde f(0, \theta) - f(0, \theta)} \abs{\sin\theta} \dtheta
&\leq  2 \int_0^\pi \abs{\tilde f(0, \theta) - f(0, \theta)} \sin\theta \dtheta
\\
& = 2\int_0^{2\alpha_\Eps} \abs{\tilde f(0, \theta) - f(0, \theta)} \sin\theta \dtheta
      + 2\int_{\pi - 2\alpha_\Eps}^\pi \abs{\tilde f(0, \theta) - f(0, \theta)} \sin\theta \dtheta
\\
&\leq 
   2\int_0^{\alpha_\Eps} \sin\theta \dtheta
\leq
   \alpha_\Eps^2 \,.
\end{align*}

\Ni (b) The exponential decay of $\frac{\del \tilde f}{\del \eta}$ follows from Remark 3.15 of \cite{WG2014}, since $\frac{\del \tilde f}{\del \eta}$ is a solution to the classical half-space equation and its incoming data satisfies~\eqref{bound:incoming-Tf-del}. The constant $\kappa_0$ is solely determined by the scattering operator and is independent of $\Eps$ as well as the incoming data. Similarly, we have the exponential decay of $\tilde f$ (with the same decay constant $\kappa_0$) such that
\begin{align} \label{decay:Tf-exp}
   \norm{e^{\kappa_0 \eta} \vpran{\tilde f - \tilde f_\infty}}_{L^\infty(\deta\dtheta)}
\leq 
   C \,,
\end{align}
where $C$ is independent of $\Eps$. To derive the exponential decay of $\frac{\del \tilde f}{\del \theta}$, we make use of the integral form of the $\tilde f$-equation~\eqref{eq:kinetic-classical-tilde}-\eqref{eq:kinetic-classical-2-tilde}:
\begin{align*}
   \tilde f(\eta, \theta) - \tilde f_\infty
= \begin{cases}
     e^{-\frac{1}{\sin\theta} \eta} \vpran{\tilde h_0 - \tilde f_\infty}
     + \int_0^\eta e^{-\frac{1}{\sin\theta} (\eta - s)} \vint{\tilde f - \tilde f_\infty}(s) \ds  \,,   & \sin\theta > 0 \,, \\[4pt]
     - \int_\eta^\infty e^{-\frac{1}{\sin\theta} (\eta - s)} \vint{\tilde f - \tilde f_\infty}(s) \ds  \,,   & \sin\theta < 0 \,.
   \end{cases}
\end{align*}
We will directly differentiate $\tilde f - \tilde f_\infty$ to show the exponential decay. For each $\theta$ such that $\sin \theta < 0$, the derivative is 
\begin{align*}
    \frac{\del \tilde f}{\del \theta}
= - \frac{\cos\theta}{\sin^2\theta}
      \int_\eta^\infty
         e^{-\frac{1}{\sin\theta} (\eta - s)} (\eta - s) \vint{\tilde f - \tilde f_\infty}(s) \ds \,.
\end{align*}
Therefore,
\begin{align*}
   \abs{\frac{\del \tilde f}{\del \theta}}
&\leq 
  C\frac{|\cos\theta|}{|\sin\theta|}
  \int_\eta^\infty 
     e^{-\frac{1}{2\sin\theta} (\eta - s)} \abs{\vint{\tilde f - \tilde f_\infty}(s)} \ds
\\
&\leq
  C\frac{|\cos\theta|}{|\sin\theta|}
  \int_\eta^\infty 
     e^{-\frac{1}{2\sin\theta} (\eta - s)} e^{-\kappa_0 s} \ds
\\
&\leq
  C e^{-\kappa_0 \eta}
  \int_\eta^\infty 
     e^{-\frac{1}{2\sin\theta} (\eta - s)} \frac{1}{|\sin\theta|} \ds
\leq 
  C e^{-\kappa_0 \eta} \,,
\end{align*}
where $C$ is independent of $\Eps$. Similarly, for each $\theta$ such that $\sin\theta > 0$, we have
\begin{align*}
    \frac{\del \tilde f}{\del \theta}
=  \frac{\cos\theta}{\sin^2\theta} \eta \, e^{-\frac{1}{\sin\theta} \eta}
      \vpran{\tilde h_0 - \tilde f_\infty}
   + e^{-\frac{1}{\sin\theta} \eta} \frac{\del \tilde h_0}{\del \theta}
    + \frac{\cos\theta}{\sin^2\theta}
      \int_0^\eta
         e^{-\frac{1}{\sin\theta} (\eta - s)} (\eta - s) \vint{\tilde f - \tilde f_\infty}(s) \ds \,.
\end{align*}
We estimate each term in $\frac{\del \tilde f}{\del \theta}$. First, since $\kappa_0 < 1$, we have 
\begin{align} \label{est:Tf-theta-1}
    \abs{\frac{\cos\theta}{\sin^2\theta} \eta \, e^{-\frac{1}{\sin\theta} \eta}
      \vpran{\tilde h_0 - \tilde f_\infty}}
    + \abs{e^{-\frac{1}{\sin\theta} \eta} \frac{\del \tilde h_0}{\del \theta}}
\leq 
    \frac{C}{\eta} e^{-\frac{\kappa_0}{\sin\theta} \eta}
    + \frac{C}{\alpha_\Eps} e^{-\kappa_0 \eta}
\leq
    \frac{C}{\alpha_\Eps} e^{- \kappa_0 \eta} \,,
\end{align}
Next, by the exponential decay of $\tilde f - \tilde f_\infty$ in~\eqref{decay:Tf-exp}, we have
\begin{align} \label{est:Tf-theta-2}
& \quad \,
   \abs{\frac{\cos\theta}{\sin^2\theta}
      \int_0^\eta
         e^{-\frac{1}{\sin\theta} (\eta - s)} (\eta - s) \vint{\tilde f - \tilde f_\infty}(s) \ds}
\leq 
  C \frac{1}{\sin\theta} 
  \int_0^\eta e^{-\frac{1}{2\sin\theta} (\eta - s)} e^{-\kappa_0 s} \ds \nn
\\
& \leq 
   \frac{C e^{-\frac{1}{2\sin\theta} \eta}}{\sin\theta} 
   \frac{1}{\kappa_0 + \frac{1}{2\sin\theta}} e^{-\vpran{\frac{1}{2\sin\theta}+\kappa_0} \eta }
\leq
   C e^{-\kappa_0 \eta} \,.
\end{align}
Combining \eqref{est:Tf-theta-1} with \eqref{est:Tf-theta-2} we obtain the exponential decay of $\frac{\del \tilde f}{\del \theta}$.
\end{proof}

Now we prove Theorem~\ref{thm:main}.
\begin{proof}[Proof of Theorem~\ref{thm:main}]
Denote $\hat f_\Eps = \tilde f - f_\Eps$. Then $\hat f$ satisfies
\begin{align} \label{eq:kinetic-difference}
   \sin\theta \frac{\del \hat f_\Eps}{\del \eta} 
   + &\hat f_\Eps - \vint{\hat f_\Eps} 
   = \frac{\Eps \psi(\Eps \eta)}{1- \Eps \eta} \cos\theta \frac{\del \hat f_\Eps}{\del \theta}
      - \frac{\Eps \psi(\Eps \eta)}{1- \Eps \eta} \cos\theta \frac{\del \tilde f}{\del \theta}
 \,, \nn
\\
   & \hat f_\Eps \big|_{\eta=0} = \tilde f(0, \theta) - f_\Eps(0, \theta) \,,
\qquad\qquad\qquad
   \sin\theta > 0 \,,
\\
   & \hat f_\Eps \to \tilde f_\infty - f_{\infty,\Eps}
\qquad \hspace{2.4cm}
   \text{\qquad as $\eta \to \infty$.} \nn
\end{align}
By Lemma~\ref{lem:modification}, we have
\begin{align*}
   \int_0^\infty \frac{\Eps \psi(\Eps \eta)}{1- \Eps \eta} 
     \abs{\int_{-\pi}^\pi \cos\theta \frac{\del \tilde f}{\del \theta} \dtheta} \deta
&\leq
    C \frac{\Eps}{\alpha_\Eps} \norm{\frac{\psi(\Eps \eta)}{1- \Eps \eta}}_{L^\infty} \int_0^\infty e^{-\kappa_0 \eta} \deta
\leq C \frac{\Eps}{\alpha_\Eps} \,,
\end{align*}
Multiply the first equation in~\eqref{eq:kinetic-difference} by $\Sgn(\hat f_\Eps)$ and integrate in $\theta$. Then
\begin{align*} 
& \quad \,
   \frac{\rd}{\rd \eta}
     \int_{-\pi}^\pi \sin\theta \abs{\hat f_\Eps} \dtheta
   - \frac{\Eps \psi(\Eps \eta)}{1- \Eps \eta} 
      \int_{-\pi}^\pi \sin\theta \abs{\hat f_\Eps} \dtheta
\leq
   \frac{\Eps \psi(\Eps \eta)}{1- \Eps \eta} 
     \abs{\int_{-\pi}^\pi \cos\theta \frac{\del \tilde f}{\del \theta} \dtheta} \,.
\end{align*}  
This implies
\begin{align*}
   \frac{\rd}{\rd \eta}
     \vpran{e^{-V_\Eps(\eta)}\int_{-\pi}^\pi \sin\theta \abs{\hat f_\Eps} \dtheta}
\leq
   e^{V_\infty}
   \frac{\Eps \psi(\Eps \eta)}{1- \Eps \eta} 
     \abs{\int_{-\pi}^\pi \cos\theta \frac{\del \tilde f}{\del \theta} \dtheta} \,,
\end{align*}
where $V_\Eps, V_\infty$ are defined in~\eqref{def:V-Eps} and~\eqref{eq:V-infty}.
Integrating in $\eta$ from $0$ to $\infty$ then gives
\begin{align*}
   - \int_{-\pi}^\pi \sin\theta \abs{\hat f_\Eps (0, \theta)} \dtheta
   \leq 
   C \frac{\Eps}{\alpha_\Eps} \,.
\end{align*}
Using the incoming data for $\hat f_\Eps$ we have
\begin{align*}
   0 \leq \int_{-\pi}^0 \abs{\sin\theta} \abs{\hat f_\Eps (0, \theta)} \dtheta
& \leq 
       C \frac{\Eps}{\alpha_\Eps}
   + \int_0^\pi \sin\theta \abs{\hat f(0, \theta)} \dtheta
\\
&\leq 
   C \frac{\Eps}{\alpha_\Eps}
   + \int_0^\pi \sin\theta \abs{\tilde f(0, \theta) - f(0, \theta)} \dtheta
\leq
   C \vpran{\frac{\Eps}{\alpha_\Eps} + \alpha_\Eps^2} \,.
\end{align*}
Let $\alpha_\Eps = \Eps^{1/3}$. Then
\begin{align*}
   0 \leq \int_{-\pi}^0 \abs{\sin\theta} \abs{\hat f_\Eps (0, \theta)} \dtheta
\leq
   C \Eps^{2/3} \,.
\end{align*}
This gives
\begin{align} \label{bound:f-tilde-f-Eps}
    \int_{-\pi}^\pi \abs{\tilde f(0, \theta) - f_\Eps(0, \theta)} \abs{\sin\theta} \dtheta
=  \int_{-\pi}^\pi \abs{\hat f_\Eps(0, \theta)} \abs{\sin\theta} \dtheta
\leq
  C \Eps^{2/3} \,.
\end{align}
Therefore by Lemma~\ref{lem:conservation} for $f_\Eps$, we have 
\begin{align*}
  \pi \abs{f_\infty - f_{\infty, \Eps}}
&= \abs{\int_{-\pi}^\pi \vpran{f_\infty - f_{\infty, \Eps}} \sin^2\theta \dtheta}
= \abs{\int_{-\pi}^\pi f(0, \theta) \sin^2\theta \dtheta
           - \int_{-\pi}^\pi f_\Eps(0, \theta)\sin^2\theta \dtheta}
    + \BigO(\Eps)
\\
& \leq
   \abs{\int_{-\pi}^\pi \vpran{\tilde f(0, \theta) - f(0, \theta)} \sin^2\theta \dtheta}
   + \abs{\int_{-\pi}^\pi \vpran{\tilde f(0, \theta) - f_\Eps(0, \theta)} \sin^2\theta \dtheta}
   + \BigO(\Eps)
\\
& \leq
   \int_{-\pi}^\pi \abs{\tilde f(0, \theta) - f(0, \theta)} \abs{\sin\theta} \dtheta
   + \int_{-\pi}^\pi \abs{\tilde f(0, \theta) - f_\Eps(0, \theta)} \abs{\sin\theta} \dtheta
   + \BigO(\Eps)
\\
& = \BigO(\Eps^{2/3}) \,,
\end{align*}
where the last inequality comes from~\eqref{ineq:infty-difference} and~\eqref{bound:f-tilde-f-Eps}.
\end{proof}

\section{Numerics} \label{Sec:numerics}
In this section we show numerical evidence of the results asserted in the previous section. Since numerical scheme is not the focus of the current paper, the details will be omitted. We refer the interested reader to~\cite{LiWang} where an implicit asymptotic preserving method for transport equation was developed. The numerical scheme for the $\varepsilon$-Milne equation is largely borrowed from there.

We briefly discuss the difficulties for numerically solving~\eqref{eq:kinetic-eps}-\eqref{eq:kinetic-eps-2} and our strategies to overcome those:
\begin{itemize}
\item Size of the domain: the equation is valid on the entire half-space domain, but it is not realistic to discretize infinite domain. Fortunately the solution decays exponentially fast, which allows us to truncate the infinite domain into a very large one: $\eta\in[0,R]$ with $R$ large. Numerically it is observed that setting $R=6$ would suffice.
\item The unknown infinite boundary condition: the well-posedness result simply implies that the solution is a constant function at $\eta = \infty$ point, but it does not suggest the value for the extrapolation length. To overcome that we borrow the idea of the shooting method but the ``shooting'' is done on both sides to match the data. More specifically, we compute $\varepsilon$-Milne equation confined in a truncated large domain $\eta\in[0,R]$ twice:
\begin{align} \label{eq:patching_up}
\begin{cases}
   \sin\theta \frac{\del f_1}{\del \eta}  - \frac{\Eps}{1- \Eps \eta} \cos\theta \frac{\del f_1}{\del \theta} + \!\!\!\!\!& f_1 - \vint{f_1} = 0 \,,\\
   f_1 \big|_{\eta=0} = h_0(\theta, \phi) \,,  &\sin\theta > 0 \,,\\
   f_1 \big|_{\eta=R} = 0 \,,  &\sin\theta < 0 \,.
   \end{cases}\quad
\begin{cases}
   \sin\theta \frac{\del f_2}{\del \eta}  - \frac{\Eps}{1- \Eps \eta} \cos\theta \frac{\del f_2}{\del \theta} + \!\!\!\!\!& f_2 - \vint{f_2} = 0 \,,\\
   f_2 \big|_{\eta=0} = 0 \,,  &\sin\theta > 0 \,,\\
   f_2 \big|_{\eta=R} = 1 \,,  &\sin\theta < 0 \,.
   \end{cases}
\end{align}
By the linearity of the $\Eps$-Milne equation,  any linear combination of $f_1$ and $f_2$ is also a solution to the same equation. There is, however only one combination that makes the solution to be approximately a constant function at $\eta = \infty$ (approximated by $R$ here). We denote it as $f_\lambda = f_1 + \lambda f_2$. Then
\begin{equation}
f_\lambda \big|_{\eta = R} = \lambda\quad\text{for}\quad\sin\theta<0\,\text{, and}\quad f_\lambda \big|_{\eta = R} = f_1\big|_{\eta = R}+ \lambda f_2\big|_{\eta = R}\quad\text{for}\quad\sin\theta>0\,.
\end{equation}
Suppose the domain is big enough with $R\gg 1$, $f_\lambda\big|_{\eta = R}$ is roughly constant in $\theta$, meaning:
\begin{equation}
f_1\big|_{\eta = R} + \lambda f_2\big|_{\eta=R} = \lambda\,.
\end{equation}
Numerically we set $\lambda = \frac{f_1}{1-f_2}\big|_{\eta = R}$. This also serves as a criterion in determining whether $R$ is indeed large enough. If $\lambda$ varies with $\theta$ then we re-run the computation on a larger domain.
\item Computing~\eqref{eq:patching_up} on a bounded domain is also
  challenging due to the singularity at $(\eta,\theta) = (0,0)$ which
  requires fine resolution. To resolve the solution, the mesh size in
  both directions have to be on the scale of $\varepsilon$:
  $\Delta\eta\sim\Delta\theta \sim \varepsilon$. The shrinking
  $\varepsilon$ induces a large linear system that is
  ill-conditioned. We borrow the idea from~\cite{LiWang}, and use a
  matrix-free scheme by performing GMRES iteration till the solution
  converges. The interested reader is referred to~\cite{LiWang} for
  details.
\end{itemize}
The scheme described above is generic and could also be applied to
$\varepsilon = 0$ case. Note in the previous work~\cite{LiLuSun2015},
we have designed a spectral method for the classical half-space (CHS)
without the spatial discretization. The spectral method is more
efficient than what is proposed here, but it does not seem to be
easily extended to treat the $\varepsilon$-Milne equation.

\subsection{Regularization of CHS}
As constructed in Lemma~\ref{lem:modification}, one can apply slight modification to the incoming data to make the solution to CHS Lipschitz. Here we show a general problem by relaxing the requirement of $\phi_1 = \phi_2$ in $(2\alpha_\varepsilon,\pi - 2\alpha_\varepsilon)$. Set the two boundary conditions as
\begin{align*}
   \phi_1 = \begin{cases}
          0, & \theta \in (0, \pi/20) \cup (\pi - \pi/20, \pi) \,,\\
          1 - \frac{9\pi}{20} \abs{x - \pi/2} \,, &  \theta \in (\pi/20, \pi-\pi/20)\,,
        \end{cases}
\end{align*}
\begin{align*}
   \phi_2 = \begin{cases}
          1, & \theta \in (0, \pi/4) \cup (\pi - \pi/4, \pi) \,,\\
          \cos\theta \,, &  \theta \in (\pi/4, \pi-\pi/4)\,.
        \end{cases}
\end{align*}
Let $f_1$ and $f_2$ be solutions to CHS with incoming data $\phi_1, \phi_2$. Their derivatives are not bounded at $\eta=0$, as shown in Figure~\ref{Figure:1} top and middle panels. By setting $\lambda_0 = \frac{1 - \vint{f_2}}{1 + \vint{f_2} - \vint{f_1}}$, the convex combination of $f_1, f_2$ given by $F = \lambda_0 f + (1 - \lambda_0) g$ is Lipschitz. This is shown in Figure~\ref{Figure:1} bottom panels of both plots.

\subsection{Computation of the $\varepsilon$-Milne problem}
We compute the $\Eps$-Milne problem and the CHS on the truncated domain with $R = 6$. We first show the truncation at $R=6$ suffices. In Figure~\ref{Figure:5} we show the 3D plot of the solutions over the entire computational domain, together with their end states. It can be seen that for both the classical half space (CHS) problem and the $\Eps$-Milne problem, at $R=6$, the solutions are approximately constant functions with variations at the order of $1e-3$. This means that the truncated domain is indeed large enough to approximate the original half space problem.

We then examine the convergence of $\Eps$-Milne problem to the CHS in different norms in terms of $\Eps$. For that we compute the $\Eps$-Milne problem with various of $\Eps$ ($\Eps = 1/25, 1/30, 1/35, 1/40, 1/45$) and measure $f_\Eps - f$ in three norms: $L^\infty(\rd\eta\rd\theta)$, $L^\infty(\rd\theta)$ at $\eta = 6$ and $L^2(\rd\eta\rd\theta)$.
\begin{itemize}
\item $L^2(\rd\eta\rd\theta)$ convergence. If $L^2$ norm is used, as $\Eps$ goes to zero, the error decreases to zero.
\item $L^\infty(\rd\theta)$ convergence at $\eta = 6$. This error decreases to zero as $\Eps$ converges to zero. This demonstrates that despite the $\Eps$-Milne problem has order $1$ difference from the CHS, the difference does not get shown at the end state.
\item $L^\infty(\rd\eta\rd\theta)$ discrepancy. With shrinking $\varepsilon$ we show the $L^\infty$ error of the solution to the $\Eps$-Milne problem and the CHS over the entire $(\eta,\theta)$ domain does not converge to zero. This provides a numerical evidence to the result shown in~\cite{WG2014}.
\end{itemize}
These results are plotted in Figure~\ref{Figure:2}.

We then look for the location of the discrepancy. The singularity of $f$ to the CHS is located at the origin where $\eta = \theta = 0$, which seems to indicate that the discrepancy takes place there. We therefore plot $f_\varepsilon - f$ along the ray of $\eta = \theta = n\varepsilon$ with $n$ being integers. It is done for various of $\Eps$. At $n=0$, the singularity takes place and we expect order 1 differences between $f_\Eps$ and $f$, but as $n$ goes bigger, we move the function away from the singularity, hoping the two solutions converge. It is indeed the case, as shown in Figure~\ref{Figure:3}. At the origin, $n=0$ and $\eta=\theta=0$, $f_\Eps-f$ is about $0.15$, but as $n$ increases, we evaluate the error function further and further away from the origin along the ray, the difference gradually disappears. Such phenomenon is universal for all $\Eps$ tested. Note that it is along this ray that the authors in \cite{WG2014} constructed the counterexample to show~\eqref{error:O-1-L-inf} instead of~\eqref{error:classical} holds.

\section{Regularization of Classical Half-Space Equations} \label{Sec:regularization}

In the second part of this paper, we will extend the first-order regularization technique used in the proof of Lemma~\ref{lem:modification} to the general case. More precisely, for any given $N \in \NN$, we use an induction proof to show how one can slightly modify the incoming data $h_0$ near $\theta=0$ so that the modified solution $\tilde f$ of the half space equation satisfies that 
\begin{align}\label{goal} 
\tilde f \in W^{N+1, \infty}(\deta\dtheta)  \,.
\end{align} 
The higher-order regularization will be useful for general geometry where the boundary of the domain has non-constant curvature.  Again without loss of generality, we assume that the original incoming data $h_0$ in equation~\eqref{eq:kinetic-classical-1} satisfies that $0 \leq h_0 \leq 1$ and is not a constant.  The main result is summarized as
\begin{thm} \label{thm:homogeneous}
Suppose the incoming data $h_0$ in equation~\eqref{eq:kinetic-classical-1} is smooth, non-constant, and satisfies that $0 \leq h_0 \leq 1$. Then for any given $\alpha_\Eps$ small enough and any $N \in \NN$, there exists $\tilde h_0(\theta) \in C^{N+1}(0, \pi)$ satisfying
\begin{align*}
    0 \leq \tilde h_0(\theta) \leq 1 \,,
\qquad
   \tilde h_0(\theta) = h_0(\theta) 
\quad
\text{on $\theta \in (2\alpha_\Eps, \pi - 2\alpha_\Eps)$}
\end{align*}
such that the solution $\tilde f$ to the half-space equation with $\tilde h_0$ as its incoming data satisfies~\eqref{goal}. Moreover, 
\begin{align} \label{bound:deriv-tilde-f-1}
     \norm{e^{\kappa_0 \eta} \frac{\del^{M+k} \tilde f}{\del\eta^M \del\theta^k}}_{L^2(\deta\dtheta)} 
= \begin{cases} 
   \BigO(\abs{\ln \alpha_\Eps}^{1/2}) \,, & M + k = 1 \,, 
\\[5pt]
   \BigO(\alpha_\Eps^{-(M+k) + 1} \abs{\ln \alpha_\Eps}^{1/2}) \,,
   & 2 \leq M+k \leq N+1
\end{cases}
\end{align}
and
\begin{align} \label{bound:deriv-tilde-f-2}
   \norm{e^{\kappa_0 \eta}\frac{\del^{M+k} \tilde f}{\del\eta^M \del\theta^k}}_{L^\infty(\deta\dtheta)} = \BigO\vpran{\alpha_\Eps^{-(M+k)}} \,,
\qquad
   1 \leq M+k \leq N+1 \,,
\end{align}
where $\kappa_0 > 0$ is the same decay constant as in Lemma~\ref{lem:modification}.
\end{thm}




\Ni {\it Notation.} In this section we use the convention that a summation $\sum_{k=k_1}^{k_2}$ is automatically zero if its upper limit $k_2$ is smaller than its lower limit $k_1$.

First we show the explicit formula for $\frac{\del^N f}{\del \eta^N}$. 

\begin{lem} \label{eq:deriv-eta-f-k}
Suppose $f$ is a smooth solution to~\eqref{eq:kinetic-classical}. Then for any $N \in \NN$,
\begin{align} \label{eq:deriv-eta-f-k}
   \frac{\del^N f}{\del \eta^N}
= (-1)^{N} 
    \frac{f - \vint{f} - \sum_{k=1}^{N-1} c_k \sin^k \theta}{\sin^{N} \theta} \,,
\qquad
    N \geq 1 \,.
\end{align}
where 
\begin{align*}
   c_1 = \vint{\frac{f - \vint{f}}{\sin\theta}} \,, \qquad c_k = \vint{\frac{f - \vint{f} - \sum_{r=1}^{k-1} c_r \sin^r \theta}{\sin^k \theta}} \,, \quad k \geq 1\,.
\end{align*}
\end{lem}
\begin{proof}
We use an induction proof. First, for $N=1,2$, we have
\begin{align*}
   \frac{\del f}{\del\eta}
= - \frac{f - \vint{f}}{\sin\theta} \,,
\end{align*}
and
\begin{align*}
   \frac{\del^2 f}{\del\eta^2}
= - \frac{\frac{\del f}{\del\eta} - \vint{\frac{\del f}{\del\eta}}}{\sin\theta}
= \frac{\frac{f - \vint{f}}{\sin\theta} - \vint{\frac{f - \vint{f}}{\sin\theta}}}{\sin\theta}
= \frac{f - \vint{f} - \sin\theta \vint{\frac{f - \vint{f}}{\sin\theta}}}{\sin^2\theta}
= \frac{f - \vint{f} - c_1\sin\theta}{\sin^2\theta} \,.
\end{align*}
Thus the cases for $N=1, 2$ are verified.  Suppose \eqref{eq:deriv-eta-f-k} holds for $N \geq 2$. Then 
\begin{align*}
   \frac{\del^{N+1} f}{\del \eta^{N+1}}
&= - \frac{\frac{\del^N f}{\del \eta^N} - \vint{\frac{\del^N f}{\del \eta^N}}}{\sin\theta}
= (-1)^{N+1} \frac{\frac{f - \vint{f} - \sum_{k=1}^{N-1} c_k \sin^k \theta}{\sin^{N} \theta} - \vint{\frac{f - \vint{f} - \sum_{k=1}^{N-1} c_k \sin^k \theta}{\sin^{N+1} \theta}}}{\sin\theta}
\\
& = (-1)^{N+1} \frac{f - \vint{f} - \sum_{k=1}^{N-1} c_k \sin^k \theta - \sin^{N}\theta \vint{\frac{f - \vint{f} - \sum_{k=1}^{N-1} c_k \sin^k \theta}{\sin^{N} \theta}}}{\sin^{N+1}\theta} 
\\
& = (-1)^{N+1} \frac{f - \vint{f} - \sum_{k=1}^{N} c_k \sin^k \theta}{\sin^{N+1}\theta}\,,
\end{align*}
which proves equation~\eqref{eq:deriv-eta-f-k} for $N+1$ thus for any $N \in \NN$.
\end{proof}

\bigskip 

\subsection*{\underline{Construction}} Functions that have $\sin^k\theta$ as the incoming data near $\theta =0$ will play a major role. Therefore, we first define some auxiliary functions. Let $0 < \alpha_\Eps < \pi/4$. 
For each $k \geq 1$, define $R_1, R_2, F_k$ as solutions to the half-space equation~\eqref{eq:kinetic-classical} with incoming data $r_1, r_2, f_k$, where
\begin{align}
   r_1 
   = \begin{cases}
       1, & \theta \in [0, \alpha_\Eps] \cup [\pi -\alpha_\Eps, \pi] \,, \\[2pt]
       h_0(\theta), & \theta \in [2 \alpha_\Eps, \pi-2\alpha_\Eps] \,,
      \end{cases}
\qquad
   r_2 
   = \begin{cases}
       0, & \theta \in [0, \alpha_\Eps] \cup [\pi -\alpha_\Eps, \pi] \,, \\[2pt]
       h_0(\theta), & \theta \in [2 \alpha_\Eps, \pi - 2\alpha_\Eps] \,,
      \end{cases}   \label{def:f-1-2}
\end{align}
and
\begin{align}
   f_k
   = \begin{cases}
       \sin^k\theta, & \theta \in [0, \alpha_\Eps] \cup [\pi -\alpha_\Eps, \pi] \,, \\[2pt]
       0, & \theta \in [2 \alpha_\Eps, \pi - 2 \alpha_\Eps] \,.
      \end{cases}
\end{align}
We also assume that
\begin{align} \label{assump:f-k}
   0 \leq r_1, r_2 \leq 1 \,,
\qquad
   0 \leq f_k \leq \alpha_\Eps^k \,,
\qquad
  r_1, r_2, f_k \in C^\infty([0, \pi]) \,.
\end{align}
Before proceeding further with the construction, we show a lemma which estimates the size of $\vint{F_k}(0)$ for each $k$.
\begin{lem} \label{lem:bound-F-k}
The functions $R_1, R_2, F_k$ satisfy that 
\begin{align*}
   0 < \vint{R_1}(0) , \vint{R_2}(0)< 1 \,,
\qquad
   \norm{F_k}_{L^{\infty}} = \BigO\vpran{\alpha_\Eps^k} \,,
\qquad
   \vint{F_k}(0) = \BigO(\alpha_\Eps^{k+1}) \,,
\qquad
   k \geq 1 \,.
\end{align*}
\end{lem}
\begin{proof}
Note that $\Disp \norm{F_k}_{L^\infty} \sim \BigO(\alpha_\Eps^k)$ is guaranteed by the maximum principle. Therefore, to obtain the desired bound for $\vint{F_k}(0)$ we only need to check the integration over $(-\pi + \alpha_\Eps, -\alpha_\Eps)$ of $F_k$. By the conservation law, we have 
\begin{align*}
    \int_{-\pi}^0 F_k \abs{\sin\theta} \dtheta
  = \int_0^\pi F_k  \sin\theta \dtheta
  \leq 
     C \int_0^{2\alpha_\Eps} \alpha_\Eps^k \sin\theta \dtheta
    + C \int_{\pi - 2\alpha_\Eps}^{\pi} \alpha_\Eps^k \sin\theta \dtheta
  = \BigO(\alpha_\Eps^{k+2}) \,,
\qquad
   F_k \geq 0 \,.
\end{align*}
Hence,
\begin{align*}
    \int_{-\pi+\alpha_\Eps}^{-\alpha_\Eps} F_k \dtheta
\leq
   \frac{1}{\sin\alpha_\Eps} \int_{-\pi+\alpha_\Eps}^{-\alpha_\Eps} F_k \abs{\sin\theta}\dtheta
\leq
   \frac{1}{\sin\alpha_\Eps} \int_{-\pi}^{0} F_k \abs{\sin\theta}\dtheta
= \BigO(\alpha_\Eps^{k+1})\,.
\end{align*}
This shows $\vint{F_k}(0) = \BigO(\alpha_\Eps^{k+1})$. The estimate $0 < \vint{R_1}(0), \vint{R_2}(0) < 1$ follows from the stability of the half-space equation stated in Lemma~\ref{lem:perturbation-1}. 
\end{proof}

Similar argument as for Lemma~\ref{lem:bound-F-k} shows
\begin{lem} \label{lem:perturbation-1}
Suppose $\Gamma$ is a solution to the half-space equation~\eqref{eq:kinetic-classical} with  incoming data  $\gamma_0 \in L^\infty(0, \pi)$.

\Ni (a) If $\gamma_0$ satisfies that
\begin{align*}
   \gamma_0 
   = \begin{cases}
       \BigO(1), & \theta \in (0, 2\alpha_\Eps) \cup (\pi -2\alpha_\Eps, \pi)\,, \\[2pt]
       0, & \theta \in (2 \alpha_\Eps, \pi - 2\alpha_\Eps) \,,
      \end{cases}
\end{align*}
then $\vint{\Gamma}(0) = \BigO(\alpha_\Eps)$.

\medskip
\Ni (b) If $\gamma_0$ satisfies that
\begin{align*}
   \gamma_0 
   = \begin{cases}
       \BigO(\alpha_\Eps^{k-1}), & \theta \in (0, 2\alpha_\Eps) \cup (\pi -2\alpha_\Eps, \pi) \,, \\[2pt]
       \BigO(\alpha_\Eps^k), & \theta \in (2 \alpha_\Eps, \pi - 2 \alpha_\Eps) \,,
      \end{cases}
\qquad
   k \geq 1 \,,
\end{align*}
then $\vint{\Gamma}(0) = \BigO(\alpha_\Eps^k)$.
\end{lem}
\begin{proof}
(a) The proof is similar to Lemma~\ref{lem:bound-F-k}. First the maximum principle gives $\Gamma = \BigO(1)$. Using the conservation law, we have
\begin{align*}
    \int_{-\pi}^0 \Gamma \abs{\sin\theta} \dtheta
  = \int_0^\pi \Gamma  \sin\theta \dtheta
  \leq 
     C \int_0^{2\alpha_\Eps} \sin\theta \dtheta
    + C \int_{\pi - 2\alpha_\Eps}^{\pi} \sin\theta \dtheta
  = \BigO(\alpha_\Eps^{2}) \,.
\end{align*}
Therefore,
\begin{align*}
    \int_{-\pi+\alpha_\Eps}^{-\alpha_\Eps} \Gamma \dtheta
\leq
   \frac{1}{\sin\alpha_\Eps} \int_{-\pi+\alpha_\Eps}^{-\alpha_\Eps} \Gamma \abs{\sin\theta}\dtheta
\leq
   \frac{1}{\sin\alpha_\Eps} \int_{-\pi}^{0} \Gamma \abs{\sin\theta}\dtheta
= \BigO(\alpha_\Eps)\,.
\end{align*}
Similar argument applied to $-\Gamma$ then gives $\vint{\Gamma}(0) = \BigO(\alpha_\Eps)$.

\medskip

\Ni (b) The proof of part (b) follows from part (a) together with the linearity and maximum principle for the classical half-space equation.
\end{proof}

The following lemma is crucial for the estimates in this section:
\begin{lem} \label{lem:int-f-0-inhomo}
Suppose $f \in L^\infty \cap L^2(\deta\dtheta)$ satisfies the half-space equation~\eqref{eq:kinetic-classical}-\eqref{eq:kinetic-classical-2} with $h_0 \in L^\infty(0 ,\pi)$. 
Then for any $\alpha_\Eps \in (0, 1)$, we have 
\begin{align} \label{bound:int-f-0}
   \abs{\vint{f}(0)}
\leq 
  C_0 \min\vpran{\alpha_\Eps \abs{\ln\alpha_\Eps}^{1/2} \norm{h}_{L^\infty}
  + |\ln \alpha_\Eps|^{1/2}\vpran{\int_{\alpha_\Eps}^{\pi-\alpha_\Eps} \sin\theta \abs{h_0}^2 \dtheta}^{1/2},  \,\, \norm{h_0}_{L^\infty}} \,,
\end{align}
and
\begin{align} \label{bound:int-f-1}
   \abs{\vint{f}(0)}
\leq 
  C_0 \min\vpran{\alpha_\Eps \norm{h}_{L^\infty}
  + \frac{1}{\alpha_\Eps} \int_{\alpha_\Eps}^{\pi-\alpha_\Eps} \sin\theta \abs{h_0} \dtheta,  \,\, \norm{h_0}_{L^\infty}} \,,
\end{align}
where $C_0$ is a generic constant.

%
\end{lem}
\begin{proof}
The bound given by $\norm{h_0}_{L^\infty}$ is due to the maximum principle. To derive the other bounds, we first note that $\abs{f}$ satisfies
 \begin{align*}
    \sin\theta \frac{\del |f|}{\del \eta}
    + |f| - \vint{|f|} \leq 0 \,,  
\end{align*}
because $\text{sgn}(f) \vint{f} \leq |\vint{f}| \leq \vint{|f|}$. Therefore,
\begin{align*}
   \frac{\rm d}{\deta} \int_{-\pi}^\pi |f| \sin\theta \dtheta \leq 0 \,.
\end{align*}
We also have the entropy bound
\begin{align*}
   \frac{\rm d}{\deta} \int_{-\pi}^\pi f^2 \sin\theta \dtheta \leq 0 \,.
\end{align*}
By $\Disp \int_{-\pi}^\pi |f_\infty| \sin\theta \dtheta = \int_{-\pi}^\pi |f_\infty|^2 \sin\theta \dtheta = 0$, we have 
\begin{align*}
    \int_{-\pi}^\pi  |f(0, \theta)| \sin\theta \dtheta  \geq 0 \,,
\qquad
    \int_{-\pi}^\pi  f^2(0, \theta) \sin\theta \dtheta  \geq 0 \,.
\end{align*}
Hence,
\begin{align*}
    \int_{-\pi}^0 |f(0, \theta)| |\sin\theta| \dtheta 
\leq 
   \int_0^\pi |h_0(\theta)| \sin\theta \dtheta \,,
\qquad
    \int_{-\pi}^0 f^2(0, \theta) |\sin\theta| \dtheta 
\leq 
   \int_0^\pi h_0^2(\theta) \sin\theta \dtheta \,.
\end{align*}
Separating $(-\pi, 0)$ into two subsets $(-\pi + \alpha_\Eps, -\alpha_\Eps)$
and $(-\alpha_\Eps, 0) \cup (-\pi, -\pi+\alpha_\Eps)$, we have
\begin{align*}
   \int_{-\pi}^0 |f(0, \theta)| \dtheta
&= \int_{-\pi+\alpha_\Eps}^{-\alpha_\Eps} |f(0, \theta)| \dtheta
   + \int_{-\pi}^{-\pi+\alpha_\Eps} |f(0, \theta)| \dtheta
   + \int_{-\alpha_\Eps}^{0} |f(0, \theta)| \dtheta
\\
& \leq
    2 \alpha_\Eps \norm{h_0}_{L^\infty}
    + \frac{C_0}{\alpha_\Eps} \int_0^\pi \sin\theta |h_0(\theta)| \dtheta 
\\
& \leq
    C_0 \alpha_\Eps \norm{h_0}_{L^\infty}
    + \frac{C_0}{\alpha_\Eps} \int_{\alpha_\Eps}^{\pi-\alpha_\Eps} \sin\theta |h_0(\theta)| \dtheta \,.
\end{align*}
Similarly, 
\begin{align*}
   \int_{-\pi}^0 |f(0, \theta)| \dtheta
&\leq
    2 \alpha_\Eps \norm{h_0}_{L^\infty}
    + \vpran{\int_{-\pi + \alpha_\Eps}^{-\alpha_\Eps} \frac{1}{\abs{\sin\theta}} \dtheta}^{1/2} \vpran{\int_0^\pi \sin\theta |h_0(\theta)|^2 \dtheta}^{1/2}
\\
&\leq
    C_0 \alpha_\Eps \abs{\ln\alpha_\Eps}^{1/2} \norm{h_0}_{L^\infty}
    + C_0 \abs{\ln\alpha_\Eps}^{1/2} \vpran{\int_{\alpha_\Eps}^{\pi-\alpha_\Eps} \sin\theta |h_0(\theta)|^2 \dtheta}^{1/2} \,,
\end{align*}
which proves the desired bounds. 
\end{proof}

Now we start constructing the approximate solution $\tilde f$. Recall that for each $N \in \NN$, we want to construct $\tilde f$ such that $\tilde f \in W^{N+1, \infty}(\deta\dtheta)$. We take the following form for the function $\tilde f$ 
\begin{align} \label{def:tilde-f-1}
    \tilde f(\eta, \theta) 
               = \lambda_N R_1 + (1 - \lambda_N) R_2 
                  + \sum_{k=1}^N c_k F_k \,,
\end{align}
where the coefficients $c_1, \ldots, c_N$ and $\lambda_N$ will be chosen so that
$\tilde f$ has the desired regularity. Note that by construction $\tilde f$ satisfies the half-space equation, and its incoming data $\tilde h_0$ differs from $h_0$ only on $[0, 2 \alpha_\Eps) \cup (\pi - 2\alpha_\Eps, \pi]$.

Assume that $c_1, \ldots, c_N$ are given and satisfy 
\begin{align} \label{bound:c-F-k}
    \abs{\sum_{k=1}^N c_k \vint{F_k}(0)} = o(1) \,,
\end{align}
which we will show a-posteriori in Theorem~\ref{thm:c-k}  for any finite $N$  and $\Eps$ small enough. 
Define
\begin{align*}
   G_1(\eta, \theta) = R_1 + \sum_{k=1}^N c_k F_k  \,,
\qquad
   G_2(\eta, \theta) = R_2 + \sum_{k=1}^N c_k F_k \,.
\end{align*}
We re-write \eqref{def:tilde-f-1} such that
\begin{equation*}
  \tilde f(\eta, \theta) = \lambda_N G_1(\eta, \theta) + (1 - \lambda_N) G_2(\eta, \theta) \,.
\end{equation*}
Then by \eqref{bound:c-F-k} and the bounds for $R_1, R_2$ in
Lemma~\ref{lem:bound-F-k}, we have
\begin{align*}
    0 < \vint{G_1}(0), \vint{G_2}(0) < 1\,.
\end{align*}
Therefore, 
\begin{align*}
    G_1(0, 0^+) -  \vint{G_1}(0)  
    = G_1(0, \pi^-) -  \vint{G_1}(0)
    = 1 - \vint{G_1}(0) > 0 \,,
\\ 
    G_2(0, 0^+) -  \vint{G_2}(0)  
    = G_2(0, \pi^-) -  \vint{G_2}(0)
    = - \vint{G_2}(0) < 0 \,.
\end{align*}
Hence there exists a constant $\lambda_N = \lambda_N(c_1, \cdots, c_N) \in (0, 1)$ such that 
\begin{align} \label{def:lambda}
   \lambda_N \vpran{G_1(0, 0^+) - \vint{G_1}(0)}
   + (1 - \lambda_N) \vpran{G_2(0, 0^+) - \vint{G_2}(0)} = 0 \,.
\end{align}
The properties of $\lambda_N$ are summarized in the following lemma:
\begin{lem} \label{lem:lambda}
Let $\lambda_N$ be defined in~\eqref{def:lambda}. Then
\begin{align} \label{eq:lambda}
   \lambda_N 
  = \bigl\langle\tilde f\bigr\rangle(0) = \vint{\lambda_N G_1 + \vpran{1-\lambda_N} G_2}(0)
  = \mu_0 + \sum_{k=1}^N c_k \mu_k \,,
\end{align}
where
\begin{align*}
   \mu_0 = \frac{\vint{R_2}(0)}{1 - \vint{R_1}(0) + \vint{R_2}(0)} \,,
\qquad 
  \mu_k = \frac{\vint{F_k}(0)}{1 - \vint{R_1}(0) + \vint{R_2}(0)} \,,
\qquad
  k \geq 1 \,.
\end{align*}
Note that by the estimates of $\vint{R_1}(0)$ and $\vint{R_2}(0)$ in Lemma~\ref{lem:bound-F-k}, all the $\mu_k$'s are well-defined.
\end{lem}
\begin{proof}
By \eqref{def:lambda}, 
\begin{align*}
   \lambda_N G_1(0, 0^+) + (1 - \lambda_N) G_2(0, 0^+)
   = \vint{\lambda_N G_1 + \vpran{1-\lambda_N} G_2}(0) \,.
\end{align*}
where $G_1(0, 0^+) = 1$ and $G_2(0, 0^+) = 0$. Hence,
\begin{align*}
   \lambda_N
&   = \vint{\lambda_N G_1 + \vpran{1-\lambda_N} G_2}(0)
   = \lambda_N \vpran{\vint{G_1}(0) - \vint{G_2}(0)} + \vint{G_2}(0)
\\
&   = \lambda_N \vpran{\vint{R_1}(0) - \vint{R_2}(0)} + \vint{G_2}(0) \,.
\end{align*}
Solving for $\lambda_N$ then gives 
\begin{align*}
   \lambda_N
   = \frac{\vint{G_2}(0)}{1 - \vint{R_1}(0) + \vint{R_2}(0)}
   = \frac{\vint{R_2}(0)}{1 - \vint{R_1}(0) + \vint{R_2}(0)}
      + \sum_{k=1}^N c_k \frac{\vint{F_k}(0)}{1 - \vint{R_1}(0) + \vint{R_2}(0)} \,,
\end{align*}
where in the last step the definition of $G_2$ is applied. 
\end{proof}

Following Lemma~\ref{lem:lambda} and the construction of $\tilde{f}$
in~\eqref{def:tilde-f-1}, we
have 
\begin{align} \label{eq:tilde-f-small-theta}
   \tilde f(0, \theta) = \bigl\langle\tilde f\bigr\rangle(0) + \sum_{k=1}^N c_k \sin^k\theta 
\qquad
  \text{for $\theta \in [0, \alpha_\Eps] \cup [\pi - \alpha_\Eps, \pi]$.}
\end{align}
Next we choose the coefficients $\{c_k\}_{k=1}^N$ such that 
$\tilde f \in W^{N+1, \infty}(\deta\dtheta)$, as in the following theorem:
\begin{thm} \label{thm:c-k}
For any given family of $R_1, R_2, F_k$, suppose $\tilde f$ is defined by \eqref{eq:tilde-f-small-theta} and $\lambda_N$ satisfies~\eqref{eq:lambda}. Then the system
\begin{align} \label{eq:c-k}
   c_1 = \vint{\frac{\tilde f - \bigl\langle\tilde f\bigr\rangle}{\sin\theta}}(0)  \,,
\qquad
   c_M = \vint{\frac{\tilde f - \bigl\langle \tilde f \bigr\rangle - \sum_{k=1}^{M-1} c_k \sin^k \theta}{\sin^M \theta}}(0) \,, \quad \text{for} \quad
    M = 2, \cdots, N \,.
\end{align}
has a unique set of $\{c_k\}_{k=1}^N$ as its solution. These $c_k$'s satisfy the bound
\begin{align} \label{bound:c-k}
   c_k = \begin{cases}
               \alpha_\Eps^{-k+1} \abs{\ln \alpha_\Eps} \,,
               & k = 1, 2 \,, \\[4pt]
               \alpha_\Eps^{-k+1} \,, & k \geq 3 
             \end{cases}
\end{align}
and
\begin{align} \label{bound:c-F-k-1}
     \abs{c_i \vint{F_i}(0)} = o(1) \,,
\qquad
     \norm{c_i F_i(0, \cdot)}_{L^\infty} = o(1) \,,
    i \geq 1 \,. 
\end{align}
Moreover,
$\tilde f$ determined by this set of
$\{c_k\}_{k=1}^N$ satisfies that $\tilde f \in W^{N+1,
  \infty}(\deta\dtheta)$.
\end{thm}
\begin{proof}
We divide the proof into four steps. \medskip

\Ni \underline{Step 1.} First we reformulate system~\eqref{eq:c-k}. Using Lemma~\ref{lem:lambda} and the definition of $\tilde f$ in~\eqref{def:tilde-f-1}, the $c_1$-equation becomes 
\begin{align*} 
    c_1 
&  =  \vint{\frac{\lambda_N R_1 + (1 - \lambda_N) R_2 
                  + \sum_{k=1}^N c_k F_k - \lambda_N}{\sin\theta}}
    = \vint{\frac{\lambda_N \vpran{R_1 - R_2 - 1} + R_2
                  + \sum_{k=1}^N c_k F_k}{\sin\theta}}
\\
& = \vint{\frac{\mu_0 \vpran{R_1 - R_2 - 1} + R_2}{\sin\theta}}
       + \sum_{k=1}^N  c_k
           \vint{\frac{\mu_k \vpran{R_1 - R_2 - 1} + F_k}{\sin\theta}} \,.
\end{align*}
For the ease of notation, we denote
\begin{align} \label{def:H-k}
    H_0 = \mu_0 \vpran{R_1 - R_2 - 1} + R_2 \,,
\qquad
    H_k = \mu_k \vpran{R_1 - R_2 - 1} + F_k
\quad \text{for} \quad k = 1, 2, \cdots, N \,.
\end{align}
Then
\begin{align} \label{eq:c-1}
    c_1 = \vint{\frac{H_0}{\sin\theta}}(0)
              + \sum_{k=1}^N c_k \vint{\frac{H_k}{\sin\theta}}(0) \,.
\end{align}
Similarly, the $c_M$-equation can be reformulated as
\begin{align} \label{eq:c-M}
   c_M 
&   = \vint{\frac{\lambda_N R_1 + (1 - \lambda_N) R_2 
                  + \sum_{k=1}^N c_k F_k - \lambda_N - \sum_{k=1}^{M-1} c_k \sin^k \theta}{\sin^M \theta}}(0) \nn
\\
& = \vint{\frac{\mu_0 \vpran{R_1 - R_2 - 1} + R_2 
                  + \sum_{k=1}^N c_k \mu_k \vpran{R_1 - R_2 - 1}
                  + \sum_{k=1}^N c_k F_k - \sum_{k=1}^{M-1} c_k \sin^k \theta}{\sin^M \theta}}(0) \,.
\end{align}
We will show that for $\Eps$ small enough, the system~\eqref{eq:c-1}-\eqref{eq:c-M} is uniquely solvable. The strategy  to solve for $c_M$'s is by inductive elimination. 


\Ni \underline{Step 2.} In this step we solve for $c_1$ in terms of $c_M$'s using~\eqref{eq:c-1}. By Lemma~\ref{lem:beta-1} which is proved later, the coefficient for $c_1$ on the right-hand side of~\eqref{eq:c-1} which is given by $\vint{\frac{H_1}{\sin\theta}}$ is of order $\BigO(\alpha_\Eps)$. Hence, for $\Eps$ small enough we can solve for $c_1$ from~\eqref{eq:c-1} and get
\begin{align} \label{soln:c-1}
    c_1 = \beta_{1,0} + \sum_{i=2}^N c_i \beta_{1, i} \,,
\qquad
    \beta_{1, i} 
    = \frac{\vint{\frac{H_i}{\sin\theta}}(0)}
              {1 - \vint{\frac{H_1}{\sin\theta}}(0)} \,,
\quad
   i = 0, 2, 3, \cdots, N \,.
\end{align}
Denote
\begin{align} \label{def:S-1}
    S_{1,i} = \frac{H_i}{\sin\theta}(0, \theta) \,, \qquad i = 0, 1, 2, 3, \cdots, N \,.
\end{align}
Then each coefficient $\beta_{1,i}$ has the form
\begin{align} \label{def:beta-1}
    \beta_{1,i} = \frac{\vint{S_{1, i}}}{1 - \vint{S_{1,1}}} \,, \qquad i = 0, 2, 3, \cdots, N \,.
\end{align}
In this notation , we have
\begin{align} \label{eq:c-1-1}
     c_1 
 = \vint{S_{1,0} + \sum_{k=1}^N c_k S_{1, k}}
     = \frac{1}{1 - \vint{S_{1, 1}}}
        \vint{S_{1,0} + \sum_{k=2}^N c_k S_{1,k}}
 = \beta_{1, 0} + \sum_{k=2}^N c_k \beta_{1, k} \,.
\end{align}


\Ni \underline{Step 3.} In this step, we derive general formulas for $c_M$ for $M \geq 2$. The formulas are inductive. We claim that if we let $S_{1,i}, \beta_{1,i}$ be defined as in~\eqref{def:S-1}--\eqref{def:beta-1}, and let
\begin{align} \label{def:S-k}
    S_{M,i} = \frac{S_{M-1, i} - \beta_{M-1, i} (1 - S_{M-1, M-1}) }{\sin\theta} \,,
\qquad
    \beta_{M, i} = \frac{\vint{S_{M, i}}}{1 - \vint{S_{M, M}}}\,,
\end{align}
for $M = 2, \cdots, N$, $i = 0, M, \cdots, N$, then
\begin{align} \label{soln:c-M}
     c_M 
&   = \vint{S_{M,0} + \sum_{k=M}^N c_k S_{M, k}}
     = \frac{1}{1 - \vint{S_{M, M}}}
        \vint{S_{M,0} + \sum_{k=M+1}^N c_k S_{M,k}} \nn
\\
&   = \beta_{M, 0} + \sum_{k=M+1}^N c_k \beta_{M, k} \,,
\qquad
     M = 2, \cdots, N -1 \,,
\\
  c_N & = \beta_{N, 0} \,. \label{soln:c-N}
\end{align}
Note that for \eqref{def:S-k} to make sense, we need to show that $\vint{S_{M,i}}(0)$ is well-defined and $\vint{S_{M, M}} \neq 1$.
These will be proved in Lemma~\ref{integrability:S-M-i} and Lemma~\ref{lem:bound-S-beta-M-i}.

We now prove \eqref{soln:c-M}-\eqref{soln:c-N} using an induction argument. First \eqref{soln:c-M} holds for $M=1$ by the definition of $S_{1, i}$ in~\eqref{def:S-1}. Suppose~\eqref{soln:c-M} holds for $M$. Then we check the equation for $c_{M+1}$, which has the form
\begin{align*}
   c_{M+1}
&= \vint{\frac{\mu_0 \vpran{R_1 - R_2 - 1} + R_2 
                  + \sum_{k=1}^N c_k \mu_k \vpran{R_1 - R_2 - 1}
                  + \sum_{k=1}^N c_k F_k - \sum_{k=1}^{M} c_k \sin^k \theta}{\sin^{M+1} \theta}}(0)
\\
& \hspace{-1cm}
       = \vint{\frac{\mu_0 \vpran{R_1 - R_2 - 1} + R_2 
                  + \sum_{k=1}^N c_k \mu_k \vpran{R_1 - R_2 - 1}
                  + \sum_{k=1}^N c_k F_k - \sum_{k=1}^{M-1} c_k \sin^k \theta}{\sin^{M+1} \theta} - \frac{c_M}{\sin\theta}}
\\
& \hspace{-1cm}
    = \vint{\frac{S_{M,0} + \sum_{k=M}^N c_k S_{M, k} - c_M}{\sin\theta}}
    = \vint{\frac{S_{M,0} + \sum_{k=M+1}^N c_k S_{M, k} + 
               c_M S_{M, M}- c_M}{\sin\theta}}
\\
& \hspace{-1cm}
    = \vint{\frac{S_{M,0} + \sum_{k=M+1}^N c_k S_{M, k} - 
              \frac{1}{1-\vint{S_{M, M}}}\vint{S_{M,0} + \sum_{k=M+1}^N c_k S_{M,k}} \vpran{1 - S_{M, M}} }{\sin\theta}}
\\
& \hspace{-1cm}
  = \vint{\frac{S_{M, 0} - \frac{\vint{S_{M,0}}}{1 - \vint{S_{M, M}}} \vpran{1 - S_{M, M}}}{\sin\theta}} 
     + \sum_{k=M+1}^N 
          c_k \vint{\frac{S_{M, k} - \frac{\vint{S_{M,k}}}{1 - \vint{S_{M, M}}} \vpran{1 - S_{M, M}}}{\sin\theta}} \,.
\end{align*}
Hence \eqref{soln:c-M} holds for $M+1$. Therefore it holds for any $M \geq 1$. 
We can then solve for $c_N, c_{N-1}, \cdots, c_1$ in order, which proves that there exists a unique set of $\{c_k\}_{k=1}^N$ such that~\eqref{eq:c-k} and~\eqref{eq:tilde-f-small-theta} hold.

\medskip

\Ni \underline{Step 4.} By the estimate of $\beta_{M,i}$ in Lemma~\ref{lem:bound-S-beta-M-i}, we have 
\begin{align*}
   \beta_{M,k} = \BigO(\alpha_\Eps^{k+2-M})= o(1) \,,
\qquad k \geq M \,.
\end{align*} 
Therefore, by \eqref{soln:c-M} and \eqref{soln:c-N}, we have
\begin{align*}
    c_1 = \BigO\vpran{\beta_{1,0}} = \BigO\vpran{\abs{\ln\alpha_\Eps}} \,,
\quad
    c_2 = \BigO\vpran{\beta_{2,0}} = \BigO\vpran{\frac{1}{\alpha_\Eps}\abs{\ln\alpha_\Eps}} \,,
\quad
    c_i = \BigO\vpran{\beta_{i,0}} = \BigO\vpran{\frac{1}{\alpha_\Eps^{i-1}}} \,,
\quad
   i \geq 3\,.
\end{align*}
Hence for $k = 1, 2$ and $i \geq 3$,
\begin{align*}
  \abs{c_k \vint{F_k}(0)} = \BigO\vpran{\alpha_\Eps^2 \abs{\ln\alpha_\Eps}} \,,
\qquad
   \abs{c_i \vint{F_i}(0)} = \BigO\vpran{\alpha_\Eps^{-i+1+i+1}} = \BigO\vpran{\alpha_\Eps^2} \,,
\\
  \norm{c_k F_k(0, \cdot)}_{L^\infty} = \BigO\vpran{\alpha_\Eps \abs{\ln\alpha_\Eps}} \,,
\qquad
   \norm{c_i F_i(0, \cdot)}_{L^\infty} = \BigO\vpran{\alpha_\Eps^{-i+1+i}} = \BigO\vpran{\alpha_\Eps} \,,
\end{align*}
which proves \eqref{bound:c-F-k-1}. 
\end{proof}

Now we prove the lemmas applied in the proof of Theorem~\ref{thm:c-k}.

\begin{lem} \label{lem:beta-1}
Let $H_k$ be defined as in~\eqref{def:H-k}. Then 

\Ni (a) the integral $\Disp\vint{\frac{H_k}{\sin\theta}}(0)$ is well-defined for all $k \geq 0$.

\Ni (b) For any $M \geq 1$, we have
\begin{align} \label{bound:H-0-sin-M}
   \frac{H_0}{\sin^M\theta}
   = \begin{cases}
       0, & \theta \in (0, \alpha_\Eps) \cup (\pi - \alpha_\Eps, \pi) \,, \\[2pt]
       \BigO\vpran{\sin^{-M} \theta}=\BigO(\alpha_\Eps^{-M}), & \theta \in (\alpha_\Eps, \pi - \alpha_\Eps) \,.
      \end{cases}
\end{align}
Moreover, if $f$ is the solution to the half-space equation with incoming data $\frac{H_0}{\sin^M\theta}$, Then
\begin{align*}
   \abs{\vint{f}(0)} 
   = \begin{cases}
     \BigO\vpran{\frac{1}{\alpha_\Eps^{M-1}}\abs{\ln \alpha_\Eps}} \,, 
        & M = 1, 2\,, \\[4pt]
     \BigO\vpran{\frac{1}{\alpha_\Eps^{M-1}}} \,, 
        & M \geq 3\,.
       \end{cases}
\end{align*}

\Ni (c) For any $M \geq 2$ and $1 \leq j \leq M-1$, we have
\begin{align} \label{bound:H-j-sin-M}
   \frac{H_j - \sin^j \theta}{\sin^M\theta}
   = \begin{cases}
       0, & \theta \in (0, \alpha_\Eps) \cup (\pi - \alpha_\Eps, \pi) \,, \\[2pt]
       \BigO\vpran{\sin^{-M+j} \theta} = \BigO(\alpha_\Eps^{-M+j}), & \theta \in (\alpha_\Eps, \pi - \alpha_\Eps) \,.
      \end{cases}
\end{align}
Moreover, if $f$ is the solution to the half-space equation with incoming data $\frac{H_j - \sin^j \theta}{\sin^M\theta}$, Then
\begin{align*}
   \abs{\vint{f}(0)} 
   = \begin{cases}
     \BigO\vpran{\alpha_\Eps^{j-M+1}} \,, 
        & M \geq 3 \,, \quad 1 \leq j \leq M-3\,, \\[4pt]
     \BigO\vpran{\frac{1}{\alpha_\Eps} \abs{\ln \alpha_\Eps}} \,,
        & M \geq 3 \,, \quad j = M-2 \,, \\[4pt]
      \BigO\vpran{\abs{\ln \alpha_\Eps}}\,, 
        & M \geq 2 \,, \quad j = M-1 \,.
       \end{cases}
\end{align*}

\Ni (d) For any $M \geq 1$ and $k \geq M$, we have
\begin{align} \label{bound:H-k-sin-M}
   \frac{H_k}{\sin^M\theta}
   = \begin{cases}
       \BigO(\alpha_\Eps^{k-M}), & \theta \in (0, 2\alpha_\Eps) \cup (\pi - 2\alpha_\Eps, \pi) \,, \\[2pt]
       \BigO(\alpha_\Eps^{k-M+1}), & \theta \in (2 \alpha_\Eps, \pi - 2 \alpha_\Eps) \,.
      \end{cases}
\end{align}
Hence by Lemma~\ref{lem:perturbation-1} if $f$ is the solution to the half-space equation with incoming data $\frac{H_k}{\sin^M\theta}$, Then
\begin{align*}
   \abs{\vint{f}}(0)
= \BigO(\alpha_\Eps^{k-M+1}) \,,
\qquad
   M \geq 1 \,,
\quad
   k \geq M \,.
\end{align*}
\end{lem}
\begin{proof}
(a) In order to prove that the integral terms $\Disp \vint{\frac{H_k}{\sin\theta}}$ are well-defined, we show that each $\Disp\frac{H_k}{\sin\theta}$ at $\eta=0$ is bounded on $(-\pi, \pi)$ for any $k \geq 0$. 
By the definitions of $R_1, R_2$, and $F_k$, each $H_k$ is a solution to the half-space equation. Moreover, by the definition of the $\mu_k$'s, we have
\begin{align*}
   \vint{H_0}(0)
= \frac{1}{1 - \vint{R_1}(0) + \vint{R_2}(0)}
    \vint{\vint{R_2} \vpran{R_1 - R_2 - 1} - \vint{R_1 - R_2 - 1} R_2} \big|_{\eta=0}
= 0 \,.
\end{align*}
and
\begin{align*}
   \vint{H_k}(0)
= \frac{1}{1 - \vint{R_1}(0) + \vint{R_2}(0)}
    \vint{\vint{F_k} \vpran{R_1 - R_2 - 1} - \vint{R_1 - R_2 - 1} F_k} \big|_{\eta=0}
= 0 \,,
\qquad 
   k \geq 1 \,.
\end{align*}
Therefore, 
\begin{align} \label{eq:H-k-sin}
   \frac{H_k(0, \theta)}{\sin\theta}
   = \frac{H_k(0, \theta) - \vint{H_k}(0)}{\sin\theta}
   = -\frac{\del H_k}{\del \eta}(0, \theta) \,,
\qquad
    k = 0, 1, \cdots, N \,.
\end{align}
By the maximal principle, since $\Disp \frac{\del H_k}{\del \eta}$ solves the half-space equation, we only need to show that the incoming data for $\Disp \frac{\del H_k}{\del \eta}$ is bounded for $\theta \in (0, \pi)$ (its bound depends on $\alpha_\Eps$). 
By the definition of $H_k, R_1, R_2, F_k$, we have
\begin{align*}
    - \frac{\del H_0}{\del \eta}(0, \theta)
    = \frac{H_0(0, \theta)}{\sin\theta}
    = \begin{cases}
        \frac{\vpran{\mu_0 \vpran{R_1 - R_2 - 1} + R_2}(0, \theta)}{\sin\theta}
    = \frac{R_2(0, \theta)}{\sin\theta} = 0 \,, 
       &  \theta \in (0, \alpha_\Eps) \cup (\pi - \alpha_\Eps, \pi) \,, \\[4pt]
       \frac{\vpran{\mu_0 \vpran{R_1 - R_2 - 1} + R_2}(0, \theta)}{\sin\theta}
       = \BigO\vpran{\frac{1}{\sin\theta}}
       = \BigO\vpran{\frac{1}{\alpha_\Eps}} \,,
       & \theta \in (\alpha_\Eps, \pi - \alpha_\Eps) \,,
       \end{cases}
\end{align*}
and for each $k \geq 1$,
\begin{align*}
     -\frac{\del H_k}{\del \eta}\Big|_{\eta=0}
    = \frac{H_k(0, \theta)}{\sin\theta}
    = \begin{cases}
        \frac{\vpran{\mu_k \vpran{R_1 - R_2 - 1} + F_k}(0, \theta)}{\sin\theta}
    = \frac{F_k(0, \theta)}{\sin\theta} 
    = \BigO\vpran{\sin^{k-1}\theta} \,, 
       &  \theta \in (0, 2\alpha_\Eps) \cup (\pi - 2\alpha_\Eps, \pi) \,, \\[4pt]
       \frac{\vpran{\mu_k \vpran{R_1 - R_2 - 1} + F_k}(0, \theta)}{\sin\theta}
       = \BigO\vpran{\frac{\alpha_\Eps^{k+1}}{\sin\theta}}
       = \BigO\vpran{\alpha_\Eps^k} \,,
       & \theta \in (2\alpha_\Eps, \pi - 2\alpha_\Eps) \,,
       \end{cases}
\end{align*}
which shows $\frac{\del H_k}{\del \eta} \in L^\infty(\deta\dtheta)$ for any $k \geq 0$. Therefore $\vint{\frac{\del H_k}{\del \eta}}(0)$ is well-defined. 
%

\medskip

\Ni (b) The bounds in \eqref{bound:H-0-sin-M} follow directly from the definition of $H_0$. If $M=1$, then by~\eqref{bound:int-f-0} in Lemma~\ref{lem:int-f-0-inhomo}, we have
\begin{align*}
    \abs{\vint{f}(0)}
\leq
  C_0 \vpran{\alpha_\Eps \abs{\ln \alpha_\Eps}^{1/2} \norm{\frac{H_0}{\sin\theta}}_{L^\infty}
  + \vpran{|\ln \alpha_\Eps|}^{1/2}\vpran{\int_{\alpha_\Eps}^{\pi-\alpha_\Eps} \sin\theta \abs{\frac{H_0}{\sin\theta}}^2 \dtheta}^{1/2}}
= \BigO\vpran{\abs{\ln \alpha_\Eps}} \,.
\end{align*}
If $M \geq 2$, then by~\eqref{bound:int-f-1} in Lemma~\ref{lem:int-f-0-inhomo}, we have
\begin{align*}
    \abs{\vint{f}(0)}
\leq
  C_0 \vpran{\alpha_\Eps \norm{\frac{H_0}{\sin\theta}}_{L^\infty}
  + \frac{1}{\alpha_\Eps} \int_{\alpha_\Eps}^{\pi-\alpha_\Eps} \sin\theta \abs{\frac{H_0}{\sin^M\theta}} \dtheta}
= \begin{cases}
     \BigO\vpran{\frac{1}{\alpha_\Eps} \abs{\ln \alpha_\Eps}} \,, & M = 2 \,, \\[4pt]
     \BigO\vpran{\frac{1}{\alpha_\Eps^{M-1}}} \,, & M \geq 3 \,.
   \end{cases}
\end{align*}

\medskip
\Ni (c) The bounds in \eqref{bound:H-j-sin-M} also directly comes from the definition of $H_j$. For the bound of $\vint{f}(0)$ when $1 \leq j \leq M-2$, we have
\begin{align*}
  \norm{\frac{H_j - \sin^j \theta}{\sin^M\theta}}_{L^\infty(0, \pi)}
  = \BigO\vpran{\alpha_\Eps^{j-M}} \,,
\qquad
   \int_{\alpha_\Eps}^{\pi-\alpha_\Eps} \sin\theta \abs{\frac{H_j - \sin^j \theta}{\sin^M \theta}} \dtheta
   = \begin{cases}
        \BigO\vpran{\alpha_\Eps^{j- M+2}} \,, & j \leq M-3 \,, \\[4pt]
        \BigO\vpran{\abs{\ln \alpha_\Eps}} \,, & j = M - 2 \,.
     \end{cases}
\end{align*}
Hence by ~\eqref{bound:int-f-1} in Lemma~\ref{lem:int-f-0-inhomo}, we have
\begin{align*}
    \abs{\vint{f}}(0) 
    = \begin{cases}
          \BigO\vpran{\alpha_\Eps^{j- M+1}} \,,
       &M \geq 3 \,, \quad   1 \leq j \leq M-2 \,, \\[4pt]
          \BigO\vpran{\frac{1}{\alpha_\Eps} \abs{\ln \alpha_\Eps}} \,,
        & M \geq 3 \,, \quad j = M - 2 \,.
       \end{cases}
\end{align*}
In the case when $j = M-1$, we have
\begin{align*}
  \norm{\frac{H_j - \sin^j \theta}{\sin^M\theta}}_{L^\infty(0, \pi)}
  = \BigO\vpran{\frac{1}{\alpha_\Eps}} \,,
\qquad
   \vpran{\int_{\alpha_\Eps}^{\pi-\alpha_\Eps} \sin\theta \abs{\frac{H_j - \sin^j \theta}{\sin^M \theta}}^2 \dtheta}^{1/2}
   = \BigO\vpran{\abs{\ln \alpha_\Eps}^{1/2}} \,.
\end{align*}
Thus we have $\abs{\vint{f}}(0) = \BigO\vpran{\abs{\ln \alpha_\Eps}}$ for $j = M - 1$.

\medskip
\Ni (d) 
First we have
\begin{align*}
  \abs{\frac{H_k(0, \theta)}{\sin^M\theta}}_{\eta=0}
=\BigO(\sin^{k-M}(\theta))
= \BigO\vpran{\alpha^{k-M}_\Eps} \,, 
\qquad
   \theta \in (0, \alpha_\Eps) \cup (\pi - \alpha_\Eps, \pi) \,.
\end{align*}
If $\theta \in (\alpha_\Eps, 2 \alpha_\Eps) \cup (\pi - 2\alpha_\Eps, \pi - \alpha_\Eps)$, then we have
\begin{align*}
   \frac{H_k(0, \theta)}{\sin^M\theta}
 = \frac{1}{1 - \vint{R_1}(0) + \vint{R_2}(0)}
    \frac{\vpran{\vint{F_k} \vpran{R_1 - R_2 - 1} - \vint{R_1 - R_2 - 1} F_k}}{\sin^M\theta}
 = \BigO(\alpha_\Eps^{k-M}) \,,
\end{align*}
since $0 \leq F_k \leq \alpha_\Eps^k$ and $\vint{F_k} = \BigO(\alpha_\Eps^{k+1})$ for $k \geq 1$.
Lastly, for $\theta \in (2\alpha_\Eps, \pi-2\alpha_\Eps)$, we have
\begin{align*}
   \frac{H_k(0, \theta)}{\sin^M\theta}
= - \frac{1}{1 - \vint{R_1}(0) + \vint{R_2}(0)}
    \frac{\vint{F_k} }{\sin^M\theta}
= \BigO(\alpha_\Eps^{k-M+1}) \,,
\end{align*}
where once again we have applied $\vint{F_k} = \BigO(\alpha_\Eps^{k+1})$. 
Hence, by Lemma~\ref{lem:perturbation-1} we have
\begin{align*}
  \abs{\vint{f}}(0) 
  = \BigO(\alpha_\Eps^{k-M}) \,,
\qquad
   k \geq 1\,.
\end{align*}
\end{proof}

In the following lemma we show that $\vint{S_{M, i}}$ is well-defined for any $M = 1, \cdots, N$ and $i = 0, M, \cdots, N$ and derive its explicit bound. 
\begin{lem} \label{integrability:S-M-i}
Let $S_{M, i}$ and $\beta_{M,i}$ be defined in~\eqref{def:S-1}, \eqref{def:beta-1},  and~\eqref{def:S-k} for each $M = 1, \cdots, N$ and $i = 0, M, \cdots, N$. Then 
\begin{itemize}
\item[(a)]
each $S_{M, i}$ is the restriction of a solution to the half-space equation at $\eta=0$.

\item[(b)] 
There exists $\eta_{M,i}^{(j)}$ with $1 \leq j \leq M-1$ such that
\begin{align} \label{eq:S-M-i-explicit}
    S_{M, i}
= \frac{H_i + \sum_{j=1}^{M-1} \eta_{M, i}^{(j)} \vpran{H_j - \sin^j \theta}}{\sin^M\theta} \,,
\qquad
    i = 0, M, \cdots, N \,, 
\quad
    M \geq 1 \,.
\end{align}
where the case with $M=1$ reduces to $\eta_{1, i}^{(0)} = 0$ or equivalently, 
\begin{align} \label{eq:S-1-i-explicit}
    S_{1, i}
   = \frac{H_i}{\sin\theta} \,.
\end{align}

\item[(c)] Each $\vint{S_{M, i}}$ is well-defined. 

%
\end{itemize}
\end{lem}
\begin{proof}
(a) The case where $M=1$ is proved in Lemma~\ref{lem:beta-1}. 
In general, we assume that $S_{M, i}$ is the restriction of a solution to the half-space equation at $\eta = 0$ and $\vint{S_{M, i}}$ is well-defined. Then by~\eqref{def:S-k},
\begin{align*}
        S_{M+1,i} = \frac{S_{M, i} - \beta_{M, i} (1 - S_{M, M}) }{\sin\theta} \Big|_{\eta=0} \,,
\qquad
    \beta_{M, i} = \frac{\vint{S_{M, i}}(0)}{1 - \vint{S_{M, M}}(0)} \,.
\end{align*}
Suppose $T_{M, i}$ is the solution to the half-space equation with $S_{M, i} = T_{M, i} \big|_{\eta=0}$ for $i=0, M, \cdots, N$. Then $T_{M, i} - \beta_{M, i} (1 - T_{M, M})$ is also a solution to the half-space equation. Moreover, by the definition of $\beta_{M, i}$, 
\begin{align*}
    \vint{S_{M, i} - \beta_{M, i} (1 - S_{M, M})} = 0 \,,
\qquad
    i = 0, M+1, \cdots, N \,.
\end{align*}
Therefore,
\begin{align*}
    S_{M+1, i}
&= \frac{S_{M, i} - \beta_{M, i} (1 - S_{M, M}) - \vint{S_{M, i} - \beta_{M, i} (1 - S_{M, M})}}{\sin\theta} \Big|_{\eta=0}
\\
& = -\frac{\del}{\del\eta} \vpran{T_{M, i} - \beta_{M, i} (1 - T_{M, M})} \Big|_{\eta=0} \,.
\end{align*}
Therefore $S_{M+1, i}$ is the restriction of the half-space solution $\Disp - \frac{\del}{\del\eta} \vpran{T_{M, i} - \beta_{M, i} (1 - T_{M, M})}$ to $\eta=0$. This finishes the induction proof. 

\medskip

\Ni (b) 
We prove~\eqref{eq:S-M-i-explicit} inductively. First, $M=1$ holds by ~\eqref{eq:S-1-i-explicit} and by setting $\eta_{1,i}^{(0)} = 0$ for $1 \leq i \leq N$. Assume that~\eqref{eq:S-M-i-explicit} holds for $M$. Then by~\eqref{def:S-k}, we have
\begin{align*}
& \quad \,
    S_{M+1,i} 
    = \frac{S_{M, i} - \beta_{M, i} (1 - S_{M, M}) }{\sin\theta} \Big|_{\eta=0}
\\
&  = \frac{H_i + \sum_{j=1}^{M-1} \eta_{M, i}^{(j)} \vpran{H_j - \sin^j \theta}
                 + \beta_{M,i} \vpran{H_M + \sum_{j=1}^{M-1} \eta_{M, M}^{(j)} \vpran{H_j - \sin^j \theta} - \sin^M \theta}}{\sin^{M+1}\theta}
\\
& = \frac{H_i + \sum_{j=1}^{M} \eta_{M+1, i}^{(j)} \vpran{H_j - \sin^j \theta}}
              {\sin^{M+1}\theta} \,,
\end{align*}
where 
\begin{align} \label{def:eta-M-i-1}
    \eta_{M+1, i}^{(j)} 
    = \begin{cases}
       \eta_{M, i}^{(j)} + \beta_{M, i} \eta_{M, M}^{(j)} \,,
        & i = 0, M+1, \cdots, N, 
\quad
    j = 1, 2, \cdots, M-1 \,, \quad M \geq 2 \,, \\[4pt]
    \beta_{M, i} \,,
     & i = 0, M+1, \cdots, N \,, \quad  M \geq 1 \,.
    \end{cases}
\end{align}
Therefore by induction ~\eqref{eq:S-M-i-explicit} holds. 

\medskip

\Ni (c) Now we use~\eqref{eq:S-M-i-explicit} to show that each $S_{M, i}$ is bounded. By its definition, we only need to check the behaviour of $S_{M,i}$ near $\theta = 0, \pi$. The case $M=1$ has already been shown in Lemma~\ref{lem:beta-1}. In general, first, if $i =0$, then
\begin{align*}
   S_{M, 0} 
 = \frac{H_0 + \sum_{j=1}^{M-1} \eta_{M, 0}^{(j)} \vpran{H_j - \sin^j \theta}}
              {\sin^{M}\theta} \,.
\end{align*}
Recall the definitions of $H_0, H_j$ in~\eqref{def:H-k}. We have
\begin{align*}
    S_{M, 0}(\theta) = 0 \,,
\qquad
    \theta \in (0, \alpha_\Eps) \cup (\pi - \alpha_\Eps, \pi) \,.
\end{align*}
By (a) and the maximum principle, 
we have that $\vint{S_{M,0}}(0)$ is well-defined for each $M \geq 1$. Similarly, 
\begin{align*}
   S_{M, i} 
 = \frac{H_i + \sum_{j=1}^{M-1} \eta_{M, i}^{(j)} \vpran{H_j - \sin^j \theta}}
              {\sin^{M}\theta}
 = \BigO(1) \,,
\qquad
   \theta \in (0, \alpha_\Eps) \cup (\pi - \alpha_\Eps, \pi) \,,
\quad
   i \geq M \,.
\end{align*}
Hence $\vint{S_{M,i}}(0)$ is well-defined for $M \geq 1$ and $i = 0, M, \cdots, N$.
\end{proof}

In the following lemma we show more explicit bounds of $\vint{S_{M, i}}$ and $\beta_{M, i}$.
\begin{lem} \label{lem:bound-S-beta-M-i}
Let $S_{M, i}$ and $\beta_{M, i}$ be defined in~\eqref{def:S-k}. Let
\begin{align*}
    \eta_{1,i}^{(0)} = 0 \,,
\qquad
    i = 0, 1, 2, \cdots, N \,.
\end{align*}
Then

\Ni (a) For $M \geq 1, i \geq M, $ and $1 \leq j \leq M-1$, we have
\begin{align} \label{bound:S-beta-M-i}
   \vint{S_{M, i}} =  \BigO(\alpha_\Eps^{i-M+1}) \,,
\qquad
  \beta_{M,i} = \BigO(\alpha_\Eps^{i-M+1}) \,,
\qquad
   \eta_{M, i}^{(j)} 
   = \BigO\vpran{\alpha_\Eps^{i - j + 1}} \,.
\end{align}
with $\eta_{1,i}^{(0)} = 0$ when $M=1$.
\medskip

\Ni (b) If $i=0$, then we have for $M = 1, 2$,
\begin{align} \label{bound:S-beta-M-0}
   S_{M, 0} =  \BigO\vpran{\frac{1}{\alpha_\Eps^{M-1}} \abs{\ln \alpha_\Eps}} \,,
\qquad
   \beta_{M, 0} = \BigO\vpran{\frac{1}{\alpha_\Eps^{M-1}} \abs{\ln \alpha_\Eps}} \,,
\qquad
   \eta_{M, 0}^{(j)} = \BigO\vpran{\alpha_\Eps^{- j +1} \abs{\ln \alpha_\Eps}} \,,
\end{align}
while for $M \geq 3$,
\begin{align} \label{bound:S-beta-M-0-1}
   S_{M, 0} =  \BigO\vpran{\frac{1}{\alpha_\Eps^{M-1}}} \,,
\qquad
   \beta_{M, 0} = \BigO\vpran{\frac{1}{\alpha_\Eps^{M-1}}} \,,
\qquad
   \eta_{M, 0}^{(j)} 
   = \begin{cases}
         \BigO\vpran{\alpha_\Eps^{-j +1} \abs{\ln \alpha_\Eps}} \,, & j= 1, 2, \\[2pt]
         \BigO\vpran{\alpha_\Eps^{-j +1}} \,, & 3 \leq j \leq M-1 \,.
       \end{cases}
\end{align}
\end{lem}
\begin{proof}
(a) We use an induction proof to verify~\eqref{bound:S-beta-M-i}. The base case $M=1$ satisfies
\begin{align*}
    S_{1,i} = \frac{H_i}{\sin\theta} \,,
\qquad
    \beta_{1, i} 
    = \frac{\vint{\frac{H_i}{\sin\theta}}(0)}
               {1- \vint{\frac{H_1}{\sin\theta}(0)}}  \,,
\qquad
   \eta_{1,i}^{(0)} = 0 \,,
\qquad
    i = 1, \cdots, N \,.
\end{align*}
By Lemma~\ref{lem:beta-1}, we have
\begin{align*}
   \vint{S_{1, i}} = \BigO(\alpha_\Eps^i) \,,
\qquad
    \beta_{1,i} = \BigO(\alpha_\Eps^i)
\qquad
    i = 1, \cdots, N \,.
\end{align*}
Thus the base case is verified. Now suppose~\eqref{bound:S-beta-M-i} holds for $M \geq 1$ and we consider the case $M+1$. First we estimate the size of $\eta_{M+1, i}^{(j)}$. By~\eqref{def:eta-M-i-1},
\begin{align*}
    \eta_{M+1, i}^{(M)} = \beta_{M, i} 
    = \BigO\vpran{\alpha_\Eps^{i-M+1}} \,,
\end{align*}
and
\begin{align*}
 \eta_{M+1, i}^{(j)} 
    = \eta_{M, i}^{(j)} + \beta_{M, i} \eta_{M, M}^{(j)}
    = \BigO\vpran{\alpha_\Eps^{i-j+1}}
       + \BigO\vpran{\alpha_\Eps^{i - M + 1 + M- j + 1}}
    = \BigO\vpran{\alpha_\Eps^{i-j+1}} \,,
\end{align*}
for $i \geq M+1$ and $ 1 \leq j \leq M-1$. Thus the estimate for $\eta_{M+1,i}^{(j)}$ in~\eqref{bound:S-beta-M-i} holds.

Using the bounds for $\eta_{M+1,i}^{(j)}$ and~\eqref{eq:S-M-i-explicit}, we can now bound $\vint{S_{M+1, i}}$. 
Since $S_{M+1, i}$ is the restriction of the half-space solution $T_{M+1,i}$ to $\eta = 0$, we can apply Lemma~\ref{lem:beta-1} together with the linearity of the half-space equation 
to get 
\begin{align*}
   \abs{\vint{S_{M+1,i}}}
= \BigO\vpran{\alpha_\Eps^{i-M}}
    + \sum_{j=1}^{M-1}\BigO\vpran{\alpha_\Eps^{i-j+1}} 
         \BigO\vpran{\alpha_\Eps^{j - M}}
= \BigO\vpran{\alpha_\Eps^{i-M}}
= \BigO\vpran{\alpha_\Eps^{i-(M+1)+1}} \,.
\end{align*}
Moreover, since $\beta_{M+1, i} = \BigO(\abs{\vint{S_{M+1,i}}})$, we also have
\begin{align*}
   \beta_{M+1,i} = \BigO\vpran{\alpha_\Eps^{i-(M+1)+1}} \,.
\end{align*}
This shows all the bounds in \eqref{bound:S-beta-M-i} holds for $M+1$, which proves that \eqref{bound:S-beta-M-i} holds for all $M \geq 1$.

\medskip

(b) Now we check the case where $i = 0$. Since the bounds in Lemma~\ref{lem:beta-1}(b) are slightly different for $M=1,2$, we first treat these two cases. If $M=1$, then
\begin{align*}
     S_{1,0} = \frac{H_0}{\sin\theta} \,,
\qquad
    \beta_{1, 0} 
    = \frac{\vint{\frac{H_0}{\sin\theta}}(0)}
               {1- \vint{\frac{H_1}{\sin\theta}(0)}}  \,,
\qquad
   \eta_{1,0}^{(0)} = 0 \,,
\end{align*}
By Lemma~\ref{lem:beta-1}, we have
\begin{align*}
   \abs{\vint{S_{1,0}}} = \BigO\vpran{\abs{\ln \alpha_\Eps}} \,,
\qquad
   \beta_{1,0} = \BigO\vpran{\abs{\vint{S_{1,0}}}} 
   = \BigO\vpran{\abs{\ln \alpha_\Eps}} \,,
\end{align*}
which proves the case when $M=1$. Next we check the case where $M=2$. In this case, we have
\begin{align*}
   \eta_{2,0}^{(1)} = \beta_{1,0} = \BigO\vpran{\abs{\ln \alpha_\Eps}}  \,.
\end{align*}
Recall that
\begin{align*}
    S_{2, 0} = \frac{H_0 + \eta_{2,0}^{(1)} (H_1 - \sin\theta)}{\sin^2\theta} \,.
\end{align*}
Then by Lemma~\ref{lem:beta-1}, we have
\begin{align*}
   \abs{\vint{S_{2,0}}}
= \BigO\vpran{\frac{1}{\alpha_\Eps} \abs{\ln \alpha_\Eps}}
   + \BigO\vpran{\abs{\ln \alpha_\Eps}}
      \BigO\vpran{\abs{\ln \alpha_\Eps}}
= \BigO\vpran{\frac{1}{\alpha_\Eps} \abs{\ln \alpha_\Eps}} \,.
\end{align*}
This further gives
\begin{align*}
\beta_{2,0} = \BigO\vpran{\frac{1}{\alpha_\Eps} \abs{\ln \alpha_\Eps}} \,.
\end{align*}
Now we use induction to prove the case when $M \geq 3$. The base case is $M=3$, which by~\eqref{def:eta-M-i-1} satisfies
\begin{align*}
   \eta_{3,0}^{(2)} = \beta_{2,0}
  = \BigO\vpran{\frac{1}{\alpha_\Eps} \abs{\ln \alpha_\Eps}} \,,
\quad
&  \eta_{3,0}^{(1)}
  = \eta_{2, 0}^{(1)} + \beta_{2, 0} \eta_{2, 2}^{(1)}
  = \BigO\vpran{\abs{\ln \alpha_\Eps}}
     + \BigO\vpran{\frac{1}{\alpha_\Eps} \abs{\ln \alpha_\Eps}}
        \BigO\vpran{\alpha_\Eps}
   = \BigO\vpran{\abs{\ln \alpha_\Eps}} \,.
\end{align*}
Recall that
\begin{align*}
   S_{3,0}
= \frac{H_0 + \eta_{3,0}^{(1)} (H_1 - \sin\theta)
           + \eta_{3,0}^{(2)} (H_2 - \sin^2\theta)}{\sin^3\theta} \,.
\end{align*}
By Lemma~\ref{lem:beta-1}, we have
\begin{align*}
    \abs{\vint{S_{3,0}}}
= \BigO\vpran{\frac{1}{\alpha_\Eps^2}}
   + \BigO\vpran{\abs{\ln \alpha_\Eps}}
      \BigO\vpran{\frac{1}{\alpha_\Eps} \abs{\ln \alpha_\Eps}}
   + \BigO\vpran{\frac{1}{\alpha_\Eps}\abs{\ln \alpha_\Eps}}
      \BigO\vpran{\abs{\ln \alpha_\Eps}}
= \BigO\vpran{\frac{1}{\alpha_\Eps^2}} \,.
\end{align*}
This implies
\begin{align*}
   \beta_{3,0} = \BigO\vpran{\frac{1}{\alpha_\Eps^2}} \,.
\end{align*}
Therefore~\eqref{bound:S-beta-M-0-1} holds for $M=3$. Assume that~\eqref{bound:S-beta-M-0-1} holds for $M \geq 3$. Then for $M+1$ we~ have
\begin{align*}
  \eta_{M+1,0}^{(M)} = \beta_{M,0} = \BigO\vpran{\frac{1}{\alpha_\Eps^{M-1}}}
  = \BigO\vpran{\alpha_\Eps^{-M + 1}} \,.
\end{align*}
For $j = 1, \cdots, M-1$, it holds that
\begin{align*}
    \eta_{M+1, 0}^{(j)} 
= \eta_{M, 0}^{(j)} + \beta_{M, 0} \eta_{M, M}^{(j)}
= \BigO\vpran{\alpha_\Eps^{-j+1}}
      + \BigO\vpran{\frac{1}{\alpha_\Eps^{M-1}} \cdot \alpha_\Eps^{M-j+1}}
= \BigO\vpran{\alpha_\Eps^{-j+1}} \,,
\qquad j \geq 3 \,,
\\
    \eta_{M+1, 0}^{(j)} 
= \eta_{M, 0}^{(j)} + \beta_{M, 0} \eta_{M, M}^{(j)}
= \BigO\vpran{\alpha_\Eps^{-j+1} \abs{\ln \alpha_\Eps}}
      + \BigO\vpran{\frac{1}{\alpha_\Eps^{M-1}} \cdot \alpha_\Eps^{M-j+1}}
= \BigO\vpran{\alpha_\Eps^{-j+1} \abs{\ln \alpha_\Eps}} \,,
\quad j = 1, 2 \,.
\end{align*}
This verifies~\eqref{bound:S-beta-M-0-1} for $\eta_{M+1,0}^{(j)}$. Now we check the size of $S_{M+1,0}$. By~\eqref{eq:S-M-i-explicit}, we have
\begin{align*}
    S_{M+1, 0}
&= \frac{H_0 + \sum_{j=1}^{M} \eta_{M+1, 0}^{(j)} \vpran{H_j - \sin^j \theta}}{\sin^{M+1}\theta}
\\
& \hspace{-0.5cm}
   = \frac{H_0}{\sin^{M+1}\theta}
     + \sum_{j=1}^{2} \frac{\eta_{M+1, 0}^{(j)} \vpran{H_j - \sin^j \theta}}{\sin^{M+1}\theta}
     + \sum_{j=3}^{M-2} \frac{\eta_{M+1, 0}^{(j)} \vpran{H_j - \sin^j \theta}}{\sin^{M+1}\theta}
     + \sum_{j=M-1}^{M}\eta_{M+1, 0}^{(M)} \frac{\vpran{H_M - \sin^M \theta}}{\sin^{M+1}\theta}
\\
& \hspace{-0.5cm} = \BigO\vpran{\alpha_\Eps^{-M}}
      + \sum_{j=1}^2 \BigO\vpran{\alpha_\Eps^{-j+1} \abs{\ln \alpha_\Eps}} \BigO\vpran{\alpha_\Eps^{j-M}}
      + \sum_{j=3}^{M-1} \BigO\vpran{\alpha_\Eps^{-j+1}} \BigO\vpran{\alpha_\Eps^{j-M}}
 + \BigO\vpran{\alpha_\Eps^{-M+1}} \BigO\vpran{\abs{\ln \alpha_\Eps}}
\\
& \hspace{-0.5cm}
   = \BigO\vpran{\alpha_\Eps^{-M}} \,.
\end{align*}
This further implies
\begin{align*}
   \beta_{M+1, 0} = \BigO\vpran{\frac{1}{\alpha_\Eps^M}} \,.
\end{align*}
Therefore~\eqref{bound:S-beta-M-0-1} holds for $M+1$. Thus it holds for all $M \geq 1$.
%
%
\end{proof}

To prove the bound in Theorem~\ref{thm:homogeneous}, we first recall the $L^2$-bound of solutions to the half-space equation. These are classical results and one can find their proofs in \cite{WG2014}. 

\begin{lem} \label{lem:bound-L-2-half-space}
Let $f$ be the solution to the classical half-space equation with source $g$, incoming data $h_0$, and end-state~$f_\infty$. 
Then $f$ satisfies the bounds 
\begin{align*}
     \norm{e^{\kappa_0 \eta} (f - f_\infty)}_{L^2(\deta\dtheta)}^2 
\leq C \vpran{\int_0^\pi h_0^2 \sin\theta \dtheta 
       +  \norm{e^{\kappa_0 \eta}g}_{L^2(\deta\dtheta)}^2
       + \norm{e^{\kappa_0 \eta}g}_{L^\infty(\deta\dtheta)}^2} \,,
\\
     \norm{e^{\kappa_0 \eta} (f - f_\infty)}_{L^\infty(\deta\dtheta)}^2 
\leq C \vpran{\norm{h_0}_{L^\infty(0, \pi)} 
       +  \norm{e^{\kappa_0 \eta}g}_{L^2(\deta\dtheta)}^2
       + \norm{e^{\kappa_0 \eta}g}_{L^\infty(\deta\dtheta)}^2} \,,
\end{align*}
where $\kappa_0$ is the same decay constant as in Lemma~\ref{lem:modification}. 
\end{lem}

Now we can finish the proof of Theorem~\ref{thm:homogeneous}.

\begin{proof}[Proof of Theorem~\ref{thm:homogeneous}] We divide the proof in two steps. 

\medskip

\Ni \underline{Step 1. Bounds of $\Disp\frac{\del^M \tilde f}{\del \eta^M}$.} 
Since each $\Disp\frac{\del^M \tilde f}{\del \eta^M}$ is a solution to the classical half-space equation, to show its bound in either $L^\infty$ or $L^2$, we only need to study its incoming data. By Theorem~\ref{thm:c-k}, we have a unique family of $c_k$'s which gives that 
\begin{align} \label{eq:tilde-f-small-theta-recall}
   \tilde f(0, \theta) = \bigl\langle\tilde f\bigr\rangle(0) + \sum_{k=1}^N c_k \sin^k\theta 
\qquad
  \text{for $\theta \in (0, \alpha_\Eps) \cup (\pi - \alpha_\Eps, \pi)$,}
\end{align}
and the incoming data $\tilde h_0$ only differs from $h_0$ by order $\BigO(1)$ on $\theta \in (0, \alpha_\Eps) \cup (\pi-\alpha_\Eps, \pi)$. 
By our construction, we have
\begin{align} \label{eq:del-M-tilde-f-near-0}
    \frac{\del^M \tilde f}{\del\eta^M} \Big|_{\eta=0}
& = (-1)^{M} \frac{\tilde f - \bigl\langle\tilde f\bigr\rangle - \sum_{k=1}^{M-1} c_k \sin^k \theta}{\sin^{M} \theta} \nn
\\
& = (-1)^{M} \frac{\sum_{k=M}^N c_k \sin^k\theta}{\sin^{M}\theta}
  = (-1)^{M} \sum_{k=0}^{N-M} c_{k+M} \sin^k\theta \,,
\qquad
   \theta \in (0, \alpha_\Eps) \cup (\pi - \alpha_\Eps, \pi) 
\end{align}
for $M = 1, 2, \cdots, N$ and
\begin{align}  \label{eq:del-N-tilde-f-near-0}
   \frac{\del^{N+1} \tilde f}{\del\eta^{N+1}} \Big|_{\eta=0} = 0 \,,
\qquad
   \theta \in (0, \alpha_\Eps) \cup (\pi - \alpha_\Eps, \pi) \,. 
\end{align}
Moreover, 
\begin{align} \label{eq:del-M-tilde-f-near-0-1}
    \frac{\del^M \tilde f}{\del\eta^M} \Big|_{\eta=0}
& = (-1)^{M} \frac{\tilde f - \bigl\langle\tilde f\bigr\rangle - \sum_{k=1}^{M-1} c_k \sin^k \theta}{\sin^{M} \theta} \nn
\\
& = (-1)^{M} \frac{\tilde f - \bigl\langle\tilde f\bigr\rangle}{\sin^{M}\theta}
      - (-1)^{M} \sum_{k=1}^{M-1} c_k \sin^{k-M}\theta \,,
\qquad
   \theta \in (\pi - \alpha_\Eps, \pi)
\end{align}
for $M = 1, 2, \cdots , N+1$.  Equations~\eqref{eq:del-M-tilde-f-near-0}-\eqref{eq:del-M-tilde-f-near-0-1} together with the bounds of $c_k$ in~\eqref{bound:c-k}, show that at $\eta=0$,
\begin{align} \label{bound:eta-M-L-infty-1}
   \norm{\frac{\del^M \tilde f}{\del\eta^M}(0, \cdot)}_{L^\infty(0, \pi)}
= \BigO\vpran{\frac{1}{\alpha_\Eps^{M}}} \,,
\qquad
   M = 1, 2, \cdots, N+1 \,.
\end{align}
In addition, 
\begin{align*}
   \int_0^\pi \abs{\frac{\del^M \tilde f}{\del\eta^M}(0, \theta)}^2 \sin\theta \dtheta
=  \int_0^{\alpha_\Eps} \abs{\frac{\del^M \tilde f}{\del\eta^M}}^2 \sin\theta \dtheta
   +  \int_{\pi-\alpha_\Eps}^\pi \abs{\frac{\del^M \tilde f}{\del\eta^M}}^2 \sin\theta \dtheta
   +    \int_{\alpha_\Eps}^{\pi-\alpha_\Eps} \abs{\frac{\del^M \tilde f}{\del\eta^M}}^2 \sin\theta \dtheta  \,.
\end{align*}
We estimate the three terms on the right-hand side respectively. Estimates for the first two terms are similar and we only show the details for the first one. By~\eqref{eq:del-M-tilde-f-near-0} and~\eqref{eq:del-M-tilde-f-near-0-1}, 
\begin{align*}
   \int_0^{\alpha_\Eps} \abs{\frac{\del^M \tilde f}{\del\eta^M}}^2 \sin\theta \dtheta
&   \leq 2\sum_{k=0}^{N-M} \int_0^{\alpha_\Eps} \abs{c_{k+M} \sin^k \theta}^2 \sin\theta \dtheta
\\
&   = \BigO\vpran{\sum_{k=0}^{N-M}\alpha_\Eps^{-2(k+M-1)+2k+2} \abs{\ln \alpha_\Eps}}
   = \BigO\vpran{\alpha_\Eps^{-2M + 4} \abs{\ln \alpha_\Eps}}
\end{align*}
and
\begin{align*}
   \int_{\alpha_\Eps}^{\pi-\alpha_\Eps} \abs{\frac{\del^M \tilde f}{\del\eta^M}}^2 \sin\theta \dtheta
&\leq 
  C \int_{\alpha_\Eps}^{\pi-\alpha_\Eps}
       \frac{1}{\sin^{2M-1} \theta} \dtheta
  + C \sum_{k=1}^{M-1} c_k^2
        \int_{\alpha_\Eps}^{\pi-\alpha_\Eps}
             \sin^{2k-2M+1} \theta \dtheta
\\
 & \hspace{-1cm}
 = \begin{cases}
    \BigO\vpran{\abs{\ln \alpha_\Eps}} 
       + \sum_{k=1}^{M-1}
           \BigO\vpran{\alpha_\Eps^{-2(k-1)+2k-2M+2} \abs{\ln \alpha_\Eps}}
    = \BigO\vpran{\abs{\ln \alpha_\Eps}} \,, & M = 1 \,.
     \\[2pt]
    \BigO\vpran{\alpha_\Eps^{-2M+2}} 
       + \sum_{k=1}^{M-1}
           \BigO\vpran{\alpha_\Eps^{-2(k-1)+2k-2M+2} \abs{\ln \alpha_\Eps}}
    = \BigO\vpran{\alpha_\Eps^{-2M+2}} \,, & M \geq 2 \,.
 \end{cases}
\end{align*}
Therefore, 
\begin{align} \label{bound:bdry-eta-M}
   \int_0^\pi \abs{\frac{\del^M \tilde f}{\del\eta^M}(0, \theta)}^2 \sin\theta \dtheta
   = \begin{cases}
        \BigO\vpran{\abs{\ln \alpha_\Eps}} \,, & M = 1 \,,\\[2pt]
        \BigO\vpran{\alpha_\Eps^{-M+1}} \,, & M \geq 2 \,.
      \end{cases}
\end{align}
By Lemma~\ref{lem:bound-L-2-half-space}, we obtain that
\begin{align} \label{bound:eta-M-L-2-expo}
     \norm{e^{\kappa_0 \eta} \frac{\del^M \tilde f}{\del\eta^M}}_{L^2(\deta\dtheta)}^2 
\leq C_{\kappa_0} \int_0^\pi \abs{\frac{\del^M \tilde f}{\del\eta^M}(0, \theta)}^2 \sin\theta \dtheta
= \begin{cases}
        \BigO\vpran{\abs{\ln \alpha_\Eps}} \,, & M = 1 \,,\\[2pt]
        \BigO\vpran{\alpha_\Eps^{-M+1}} \,, & M \geq 2 \,.
      \end{cases}
\end{align}
By ~\eqref{bound:eta-M-L-infty-1} and Lemma~\ref{lem:bound-L-2-half-space} again,
we have
\begin{align} \label{bound:eta-M-L-infty-2}
   \norm{e^{\kappa_0 \eta}\frac{\del^M \tilde f}{\del\eta^M}}_{L^\infty(\deta\dtheta)}
= \BigO\vpran{\alpha_\Eps^{-M}} \,,
\qquad
  M \geq 1 \,.
\end{align}
\Ni \underline{Step 2. Bounds of $\frac{\del^M \tilde f}{\del \theta^M}$ and mixed derivatives.}  Next, we check the regularity of $\tilde f$ with respect to $\theta$ and all the mixed derivatives. These will be based on the regularity in $\eta$ in Step 1. 
The main strategy is still the induction proof. First we check the case $M=1$. In this case, $\frac{\del \tilde f}{\del \theta}$ satisfies the equation
\begin{align*}
    \sin\theta \frac{\del}{\del \eta} \vpran{\frac{\del \tilde f}{\del \theta}}
    &+ \frac{\del \tilde f}{\del \theta} 
    = -\cos\theta \frac{\del \tilde f}{\del \eta} \,,  
\\
   \frac{\del \tilde f}{\del \theta} \big|_{\eta=0} &= \tilde h_0'(\theta) \,,
\qquad
   \theta \in (0, \pi) \,,   
\\
   \frac{\del \tilde f}{\del \theta} &\to 0 \,,
\qquad
   \text{as $\eta \to \infty$.}  
\end{align*}
The estimates related to the incoming data $\tilde h_0'(\theta)$ are as follows. First,
\begin{align} \label{bound:del-theta-0}
   \norm{\tilde h'_0}_{L^\infty(0, \pi)} 
   = \BigO\vpran{\alpha_\Eps^{-1}} \,.
\end{align}
Second, 
\begin{align} \label{bound:del-theta-1}
   \int_0^{\alpha_\Eps} 
      \abs{\tilde h'_0(\theta)}^2 \sin\theta \dtheta
\leq
  2 \sum_{k=1}^N \int_0^{\alpha_\Eps}
       c_k^2 k^2 \sin^{2k-1}\theta \cos^2 \theta \dtheta
= \BigO\vpran{\alpha_\Eps^{-2(k-1)+2k} \abs{\ln \alpha_\Eps}}
= \BigO\vpran{\alpha_\Eps^2 \abs{\ln \alpha_\Eps}}.
\end{align}
Similar estimate holds for $\theta \in (\pi - \alpha_\Eps, \pi)$. For the part where $\theta \in (\alpha_\Eps, \pi - \alpha_\Eps)$, we have
\begin{align} \label{bound:del-theta-2}
   \int_{\alpha_\Eps}^{\pi-\alpha_\Eps}
      \abs{\tilde h'_0(\theta)}^2 \sin\theta \dtheta  = \BigO(1) \,.
\end{align}
Hence,
\begin{align} \label{bound:del-theta-3}
   \int_0^\pi
      \abs{\tilde h'_0(\theta)}^2 \sin\theta \dtheta  = \BigO(1) \,.
\end{align}
Applying \eqref{bound:eta-M-L-infty-1} for $M=1$, we have
\begin{align*}
     \norm{e^{\kappa_0 \eta} \frac{\del \tilde f}{\del\theta}}_{L^2(\deta\dtheta)}
    = \BigO(\abs{\ln \alpha_\Eps}^{1/2}) \,.
\end{align*}
Therefore, 
\begin{align*}
     \norm{e^{\kappa_0 \eta} \frac{\del^{M+k} \tilde f}{\del\eta^M \del\theta^k}}_{L^2(\deta\dtheta)} 
= \BigO(\abs{\ln \alpha_\Eps}^{1/2}) \,,
\qquad
   \norm{e^{\kappa_0 \eta}\frac{\del^{M+k} \tilde f}{\del\eta^M \del\theta^k}}_{L^\infty(\deta\dtheta)} = \frac{1}{\alpha_\Eps} \,,
\qquad
   M+k = 1 \,.
\end{align*}
Hence the base case where $N = M+k=1$ is verified. 
%

For general $M \geq 1$, we have shown the bounds of $\frac{\del^M \tilde f}{\del \eta^M}$ in both $L^2$ and $L^\infty$ norms in Step 1. 
Suppose~\eqref{bound:deriv-tilde-f-1} and~\eqref{bound:deriv-tilde-f-2} hold for $N-1$ with $N \geq 2$.
Now we show that $\tilde f \in W^{N+1, \infty}(\deta\dtheta)$ and it satisfies the bounds in~\eqref{bound:deriv-tilde-f-1} and~\eqref{bound:deriv-tilde-f-2}. We use a further induction on the order of the derivative of $\theta$ for this fixed $N$. 
The base case $k=0$ holds due to~\eqref{bound:eta-M-L-2-expo} and~\eqref{bound:eta-M-L-infty-2}. Assume that the bounds~\eqref{bound:deriv-tilde-f-1} and~\eqref{bound:deriv-tilde-f-2} hold for $(M_1, k_1)$ satisfying that 
\begin{align*}
    M_1 + k_1 = M + k = N \,,
\qquad
    M_1 \leq M \,,
\qquad
    M \geq 1 \,.
\end{align*}
We then check the case $(M-1, k+1)$. The equation for $\Disp\frac{\del^N \tilde f}{\del \eta^{M-1}\del \theta^{k+1}}$ is
\begin{align*}
    \sin\theta \frac{\del}{\del \eta} \vpran{\frac{\del^N \tilde f}{\del \eta^{M-1}\del \theta^{k+1}}}
    + &\frac{\del^N \tilde f}{\del \eta^{M-1}\del \theta^{k+1}} 
   = G_{M-1,k+1} \,,  
\\
   \frac{\del^N \tilde f}{\del \eta^{M-1}\del \theta^{k+1}} \big|_{\eta=0} 
   &= \frac{\del^{k+1}}{\del\theta^{k+1}}
         \vpran{\frac{\del^{M-1} \tilde f}{\del\eta^{M-1}}}\,,
\qquad
   \theta \in (0, \pi) \,,   
\\
   \frac{\del^N \tilde f}{\del \eta^{M-1}\del \theta^{k+1}} &\to 0 \,,
\qquad
   \text{as $\eta \to \infty$.}  
\end{align*}
where
\begin{align*}
   G_{M-1, k+1}
= \frac{\del^{k+1}}{\del\theta^{k+1}}
    \vpran{\sin\theta \frac{\del^{M} \tilde f}{\del \eta^{M}}}
   - \sin\theta \frac{\del}{\del \eta} \vpran{\frac{\del^N \tilde f}{\del \eta^{M-1}\del \theta^{k+1}}} \,,
\end{align*}
where by the induction assumptions $G_{M-1, k+1}$ is bounded by
\begin{align} \label{bound:G-M-k}
     \norm{e^{\kappa_0 \eta} G_{M-1,k+1}}_{L^2(\deta\dtheta)} 
= \BigO(\alpha_\Eps^{-N + 1} \abs{\ln \alpha_\Eps}^{1/2}) \,.
\qquad
   \norm{e^{\kappa_0 \eta} G_{M-1,k+1}}_{L^\infty(\deta\dtheta)} = \BigO\vpran{\alpha_\Eps^{-N}} \,.
\end{align}
Meanwhile, the incoming data satisfies
\begin{align*}
    \frac{\del^{k+1}}{\del\theta^{k+1}}
         \vpran{\frac{\del^{M-1} \tilde f}{\del\eta^{M-1}}}
& = (-1)^{M-1}\frac{\del^{k+1}}{\del\theta^{k+1}}
       \vpran{\frac{\tilde f - \bigl\langle\tilde f\bigr\rangle - \sum_{i=1}^{M-2} c_i \sin^i \theta}{\sin^{M-1} \theta}}
\\
& = (-1)^{M-1}\frac{\del^{k+1}}{\del\theta^{k+1}}
       \vpran{\sum_{i=0}^{k+1} c_{i+M-1} \sin^i \theta} \,,
\qquad
   \theta \in (0, \alpha_\Eps) \cup (\pi-\alpha_\Eps, \pi) \,.
\end{align*}
Therefore,
\begin{align*}
  \int_0^{\alpha_\Eps}
     \abs{\frac{\del^{k+1}}{\del\theta^{k+1}}
         \vpran{\frac{\del^{M-1} \tilde f}{\del\eta^{M-1}}}}^2 \sin\theta \dtheta
= \BigO\vpran{c_N^2 \alpha_\Eps^2 } 
= \BigO\vpran{\alpha_\Eps^{-2N+4} \abs{\ln \alpha_\Eps}^{2}}\,.
\end{align*}
Similar estimate holds for the integration on $(\pi-\alpha_\Eps, \pi)$. For the part where $\theta \in (\alpha_\Eps, \pi - \alpha_\Eps)$, we have
\begin{align*}
    \frac{\del^{k+1}}{\del\theta^{k+1}}
         \vpran{\frac{\del^{M-1} \tilde f}{\del\eta^{M-1}}}
& = (-1)^{M-1}\frac{\del^{k+1}}{\del\theta^{k+1}}
       \vpran{\frac{\tilde f - \bigl\langle\tilde f\bigr\rangle - \sum_{i=1}^{M-2} c_i \sin^i \theta}{\sin^{M-1} \theta}}
\\
& \hspace{-3cm}
    = (-1)^{M-1}\frac{\del^{k+1}}{\del\theta^{k+1}}
       \vpran{\frac{\tilde f - \bigl\langle\tilde f\bigr\rangle}{\sin^{M-1} \theta}}
     - (-1)^{M-1}\frac{\del^{k+1}}{\del\theta^{k+1}}
       \vpran{\sum_{i=1}^{M-2} c_{i} \sin^{i-M+1} \theta} \,,
\end{align*}
Therefore,
\begin{align*}
& \quad \,
  \int_{\alpha_\Eps}^{\pi-\alpha_\Eps}
     \abs{\frac{\del^{k+1}}{\del\theta^{k+1}}
         \vpran{\frac{\del^{M-1} \tilde f}{\del\eta^{M-1}}}}^2 \sin\theta \dtheta
\\
& \leq C \int_{\alpha_\Eps}^{\pi - \alpha_\Eps}
                \frac{1}{\sin^{2N-1} \theta} \dtheta
          + C c_1 \int_{\alpha_\Eps}^{\pi - \alpha_\Eps}
                \frac{1}{\sin^{2N-3} \theta} \dtheta
= \BigO\vpran{\alpha_\Eps^{-2N+2} \abs{\ln \alpha_\Eps}}\,.
\end{align*}
This gives
\begin{align*}
  \int_0^\pi
     \abs{\frac{\del^{k+1}}{\del\theta^{k+1}}
         \vpran{\frac{\del^{M-1} \tilde f}{\del\eta^{M-1}}}}^2 \sin\theta \dtheta
= \BigO\vpran{\alpha_\Eps^{-2N+2} \abs{\ln \alpha_\Eps}}\,,
\end{align*}
Combining with~\eqref{bound:G-M-k}, we have
\begin{align*}
    \norm{e^{\kappa_0 \eta}\frac{\del^{N} \tilde f}{\del\eta^{M-1} \del \theta^{k+1}}}_{L^2(\deta\dtheta)}
 = \BigO\vpran{\frac{1}{\alpha_\Eps^{N -1}} \abs{\ln \alpha_\Eps}^{1/2}} \,,
\qquad
    \norm{e^{\kappa_0 \eta} \frac{\del^{N} \tilde f}{\del\eta^{M-1} \del \theta^{k+1}}}_{L^\infty(0, \pi)}
 = \BigO\vpran{\frac{1}{\alpha_\Eps^{N}}} \,,
\end{align*}
where $M+k = N$. This proves the induction for $k+1$ and for any arbitrary $N \in \NN$. We thereby finish the proof of Theorem~\ref{thm:homogeneous}.
\end{proof}

\bibliography{kinetic}

\begin{figure} 
\includegraphics[width = 0.4\textwidth,height = 0.2\textheight]{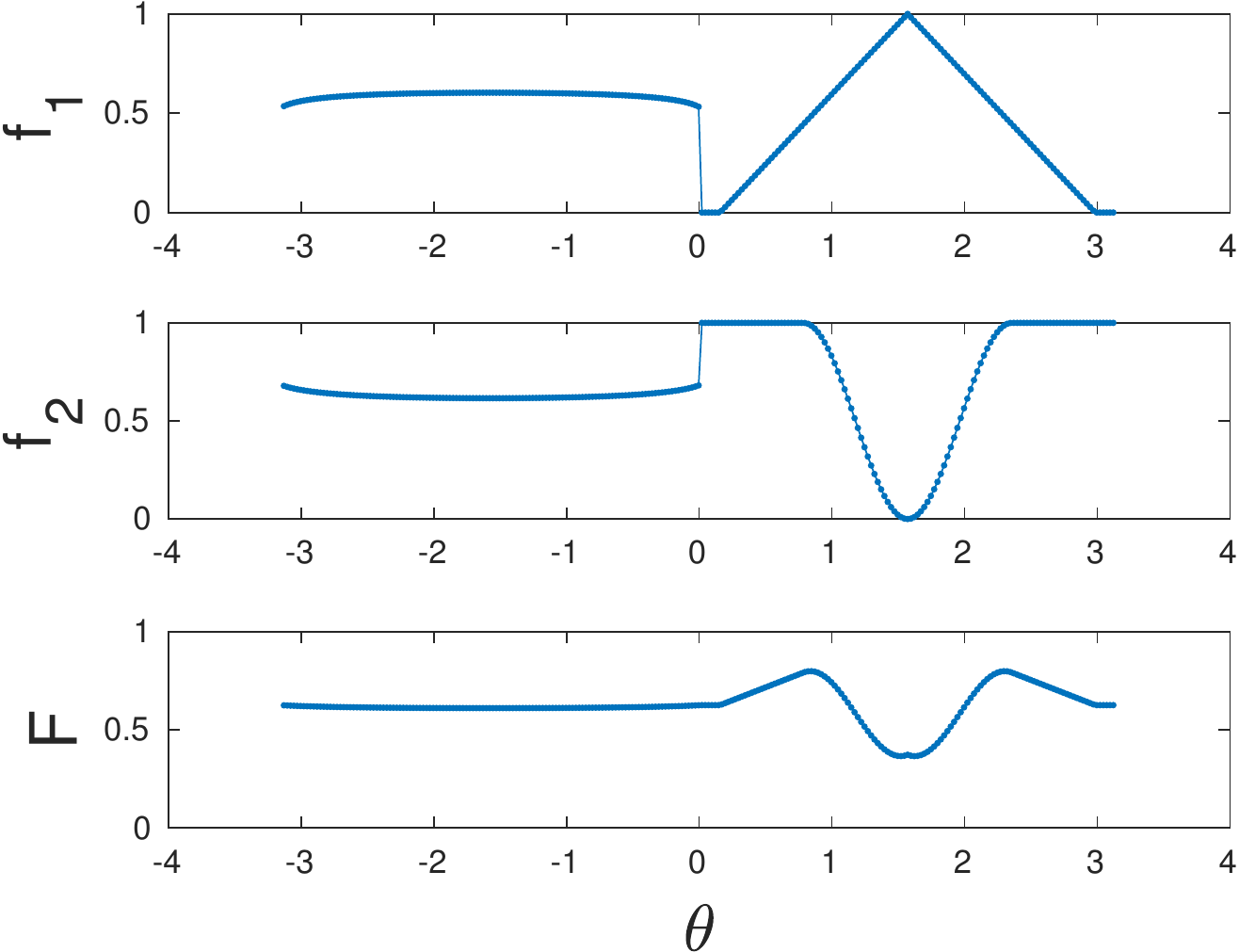}
\quad
\includegraphics[width = 0.4\textwidth,height = 0.21\textheight]{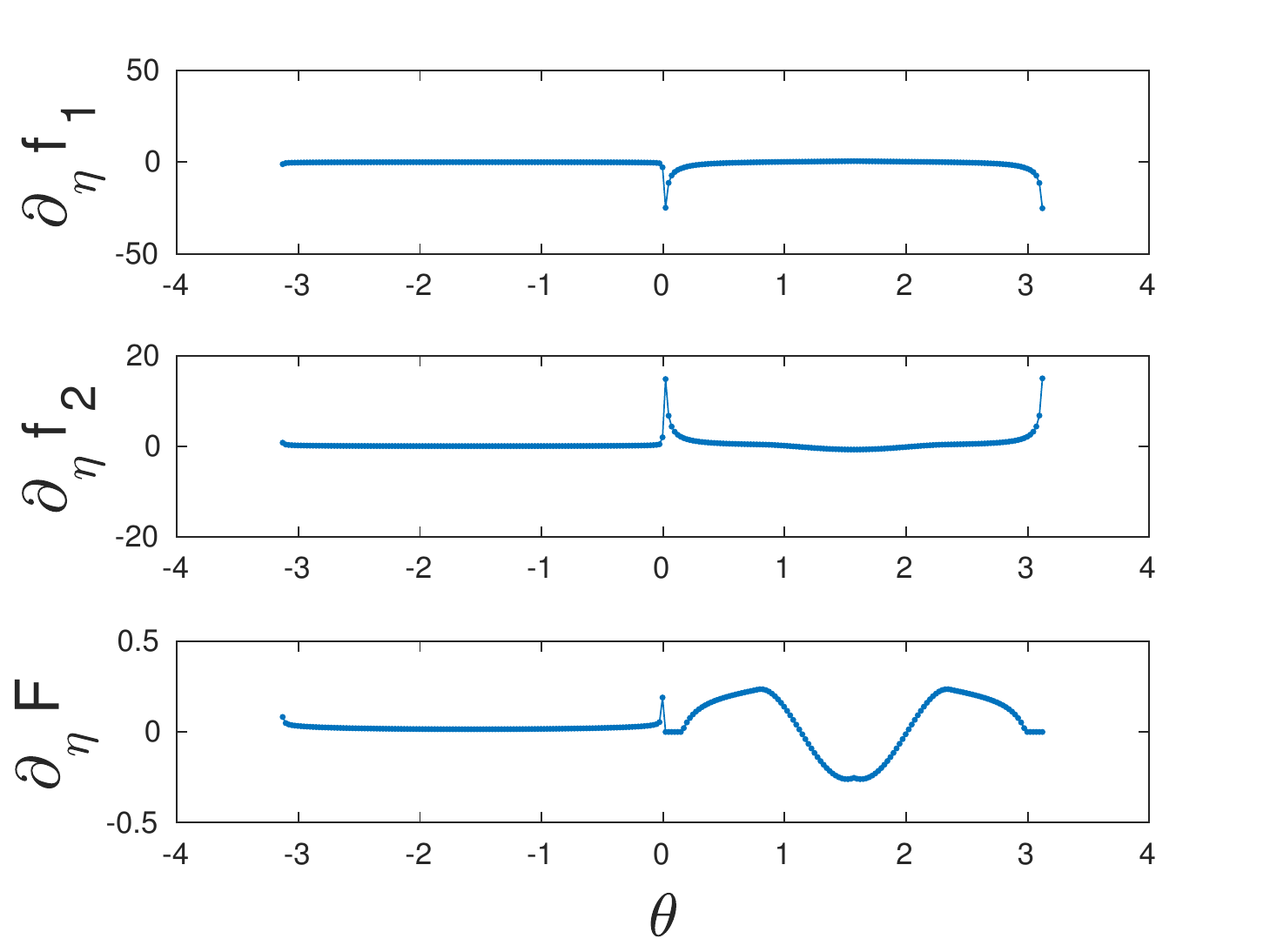} 
\caption{The plot on the left panel shows $f_1$, $f_2$, and their convex combination $F$, and on the right we show that $\del_\eta f_1$ and $\del_\eta f_2$ blows up while $\del_\eta F$ is still bounded.}
\label{Figure:1}
\end{figure}

\begin{figure} 
\includegraphics[width = 0.4\textwidth]{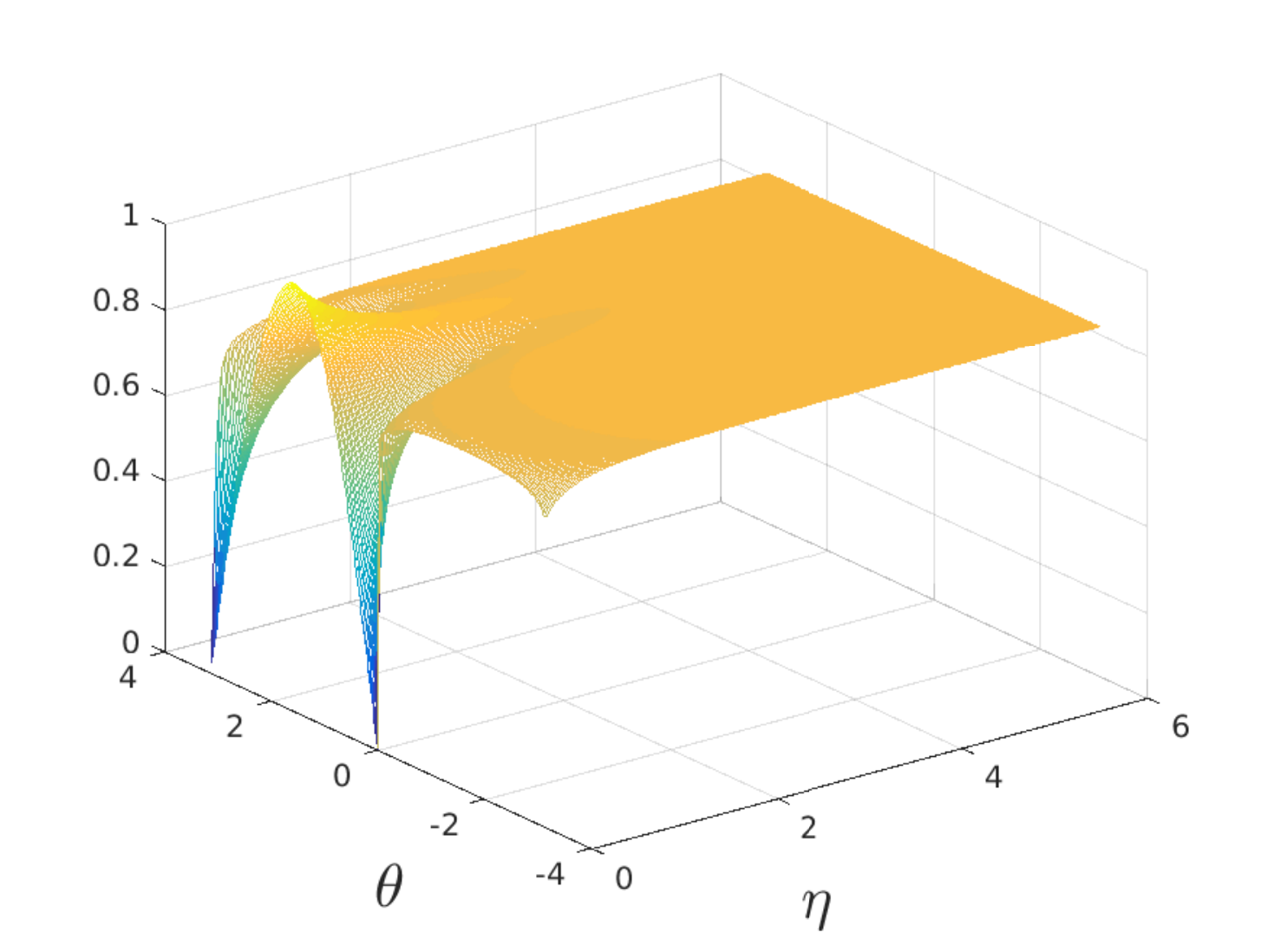} 
\qquad
\includegraphics[width = 0.4\textwidth]{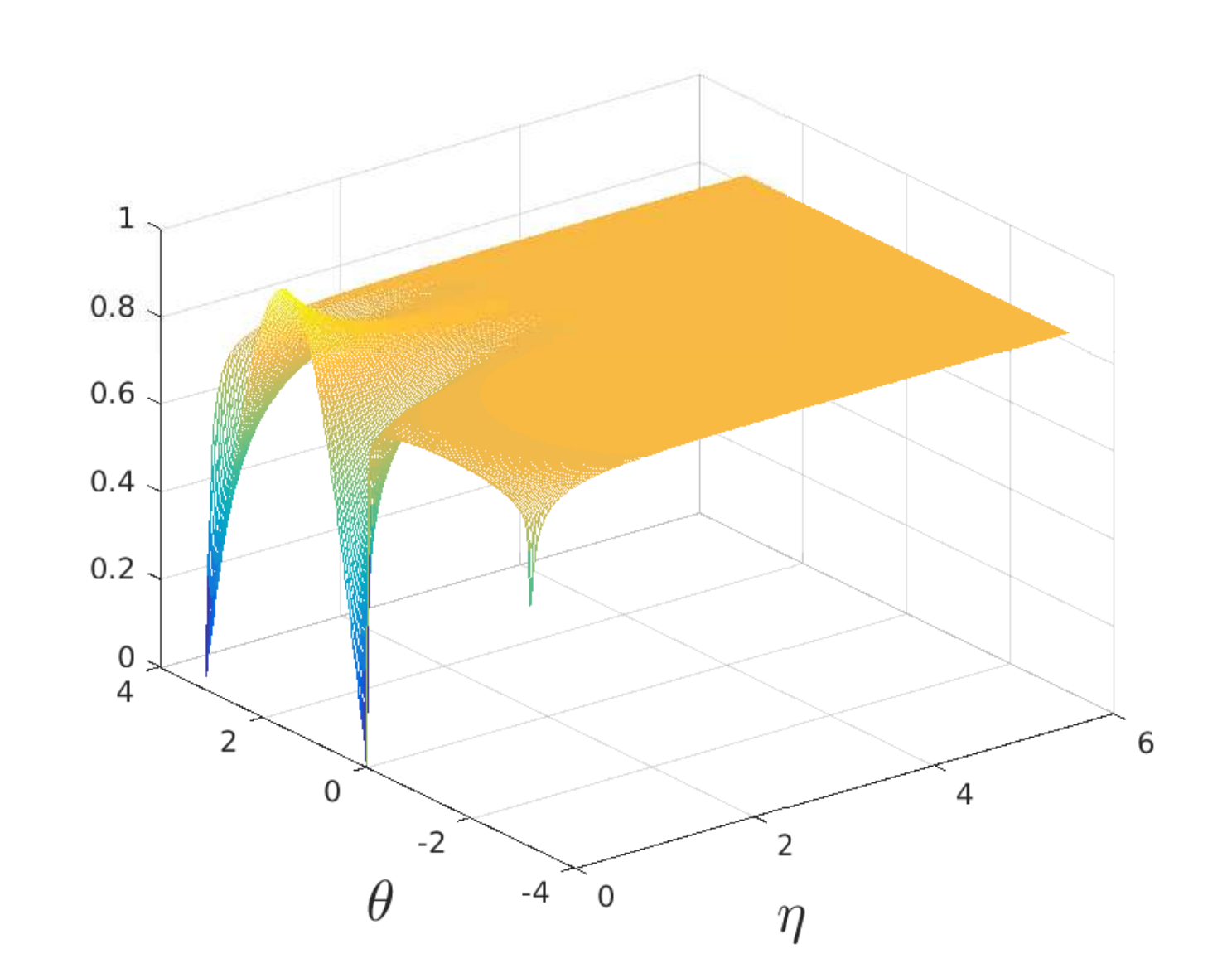}\\
\includegraphics[width = 0.4\textwidth]{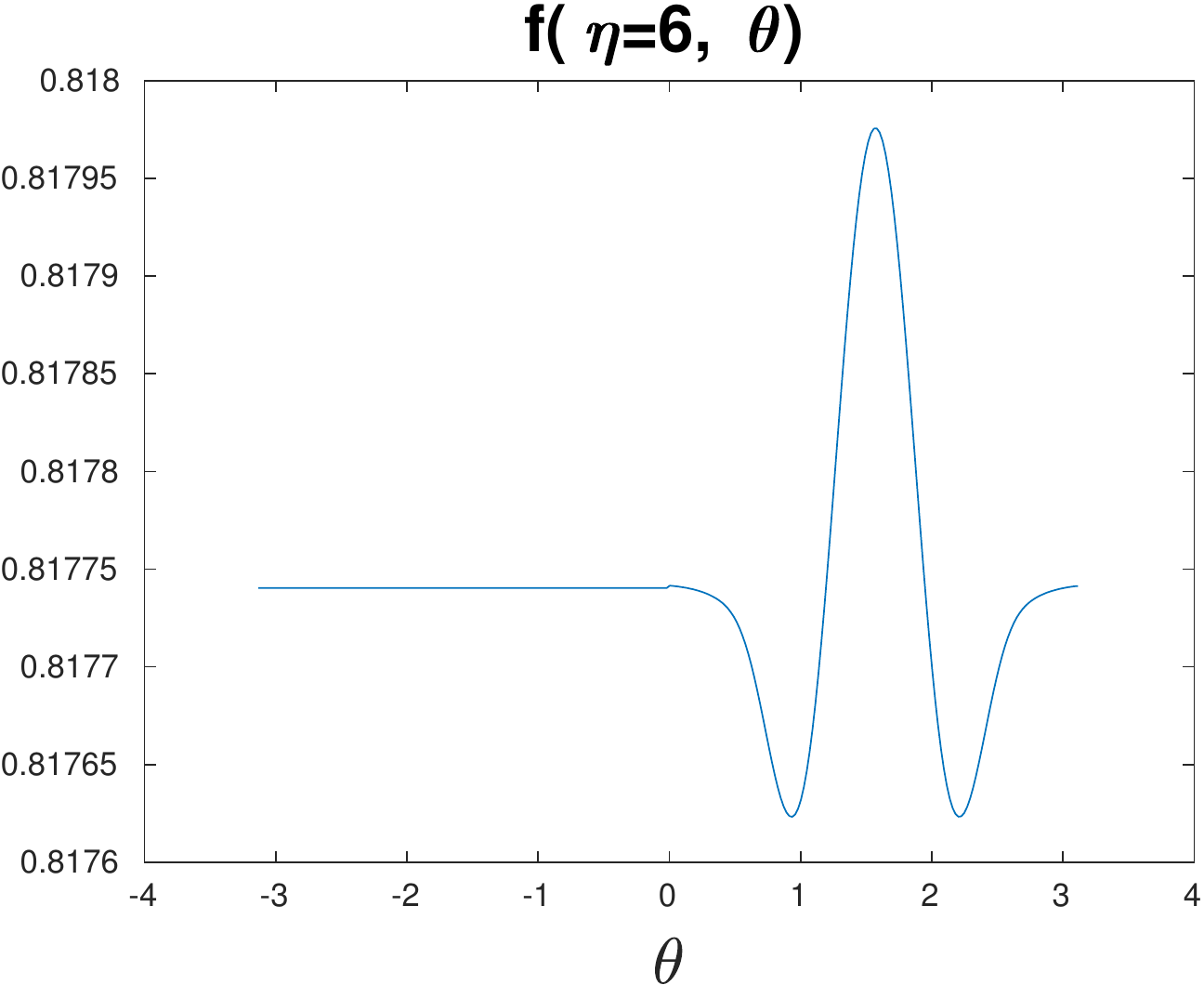}
\qquad
\includegraphics[width = 0.4\textwidth]{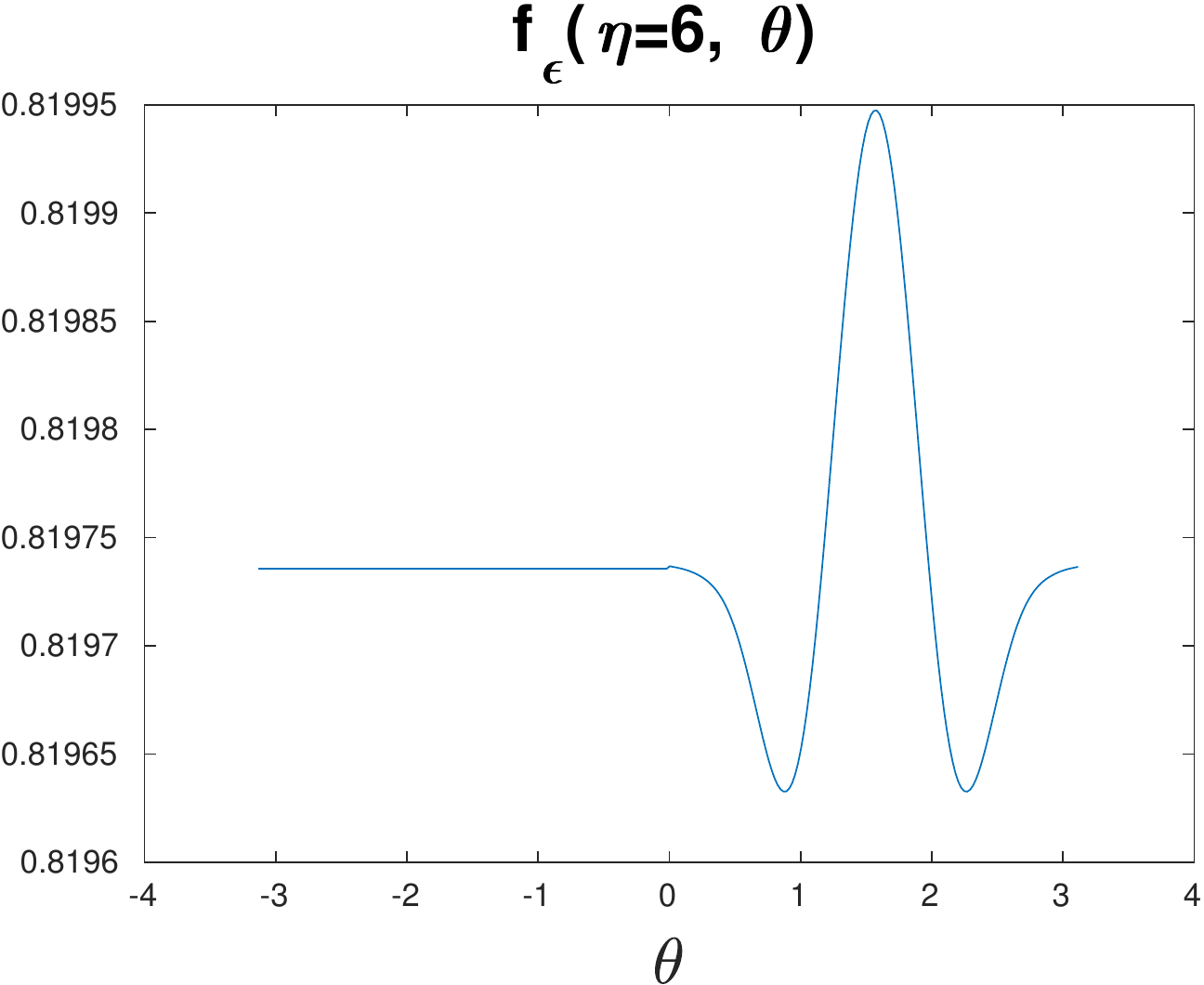} 
\caption{The two plots at the top show the solution to the classical half-space equation (left) and that to the $\Eps$-Milne equation (right). The two plots at the bottom demonstrate that at $\eta=6$, the two solutions (CHS on the left and $\Eps$-Milne equation on the right) are approximately constants, and thus truncating the domain at $\eta=6$ suffices. Here $\Eps=1/45$.}
\label{Figure:5}\
\end{figure}

\begin{figure} 
\includegraphics[width = 0.3\textwidth]{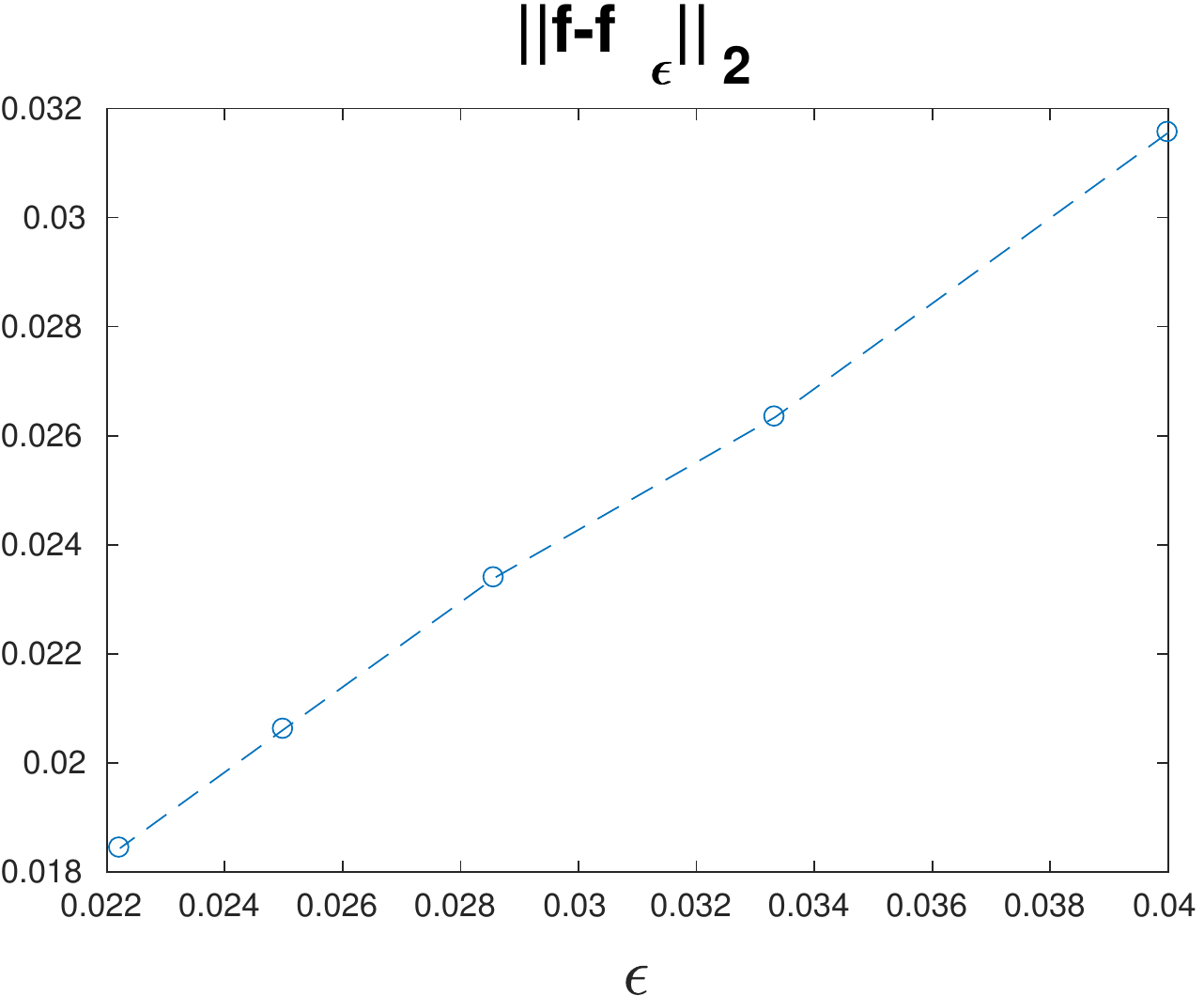}
\quad
\includegraphics[width = 0.3\textwidth]{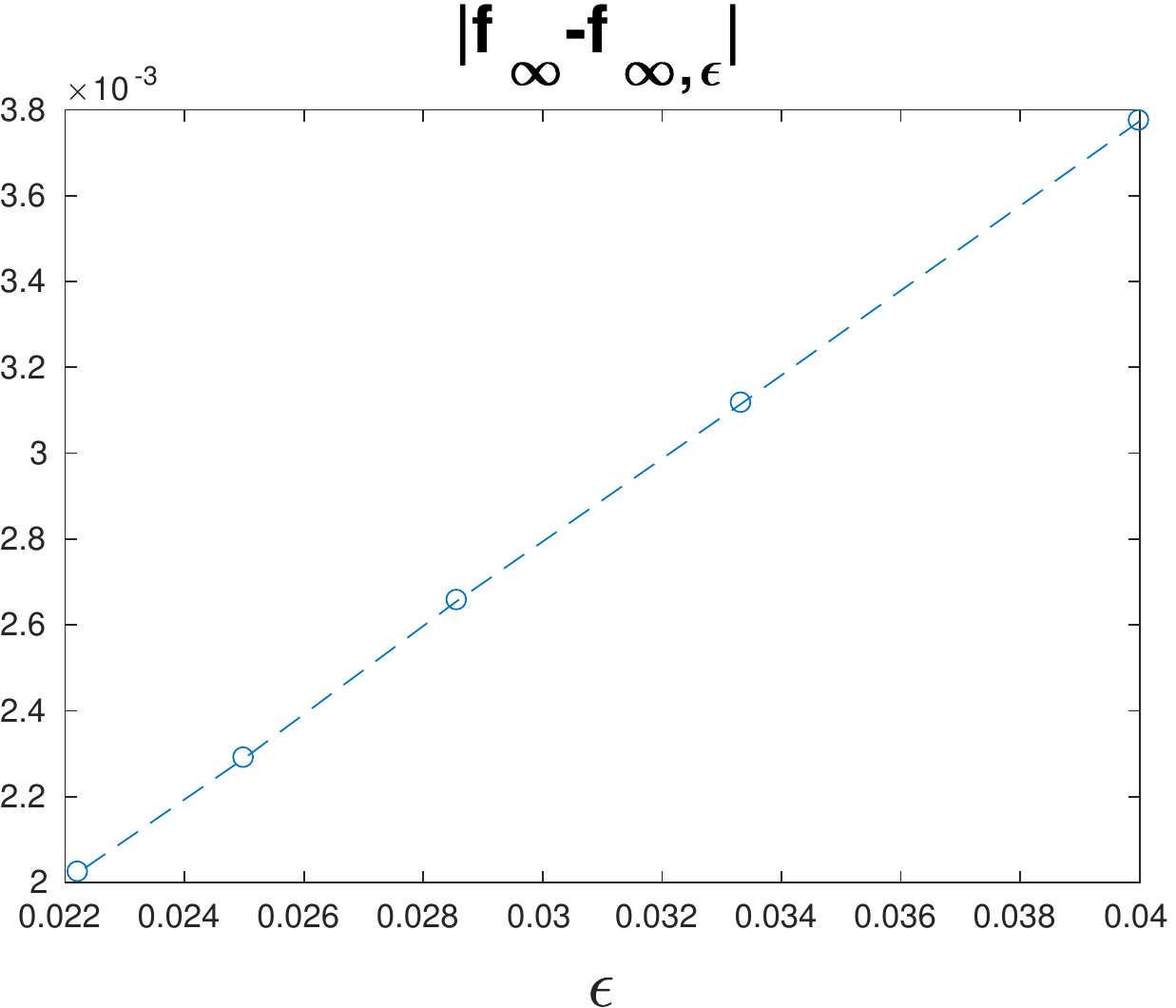} 
\quad
\includegraphics[width = 0.3\textwidth]{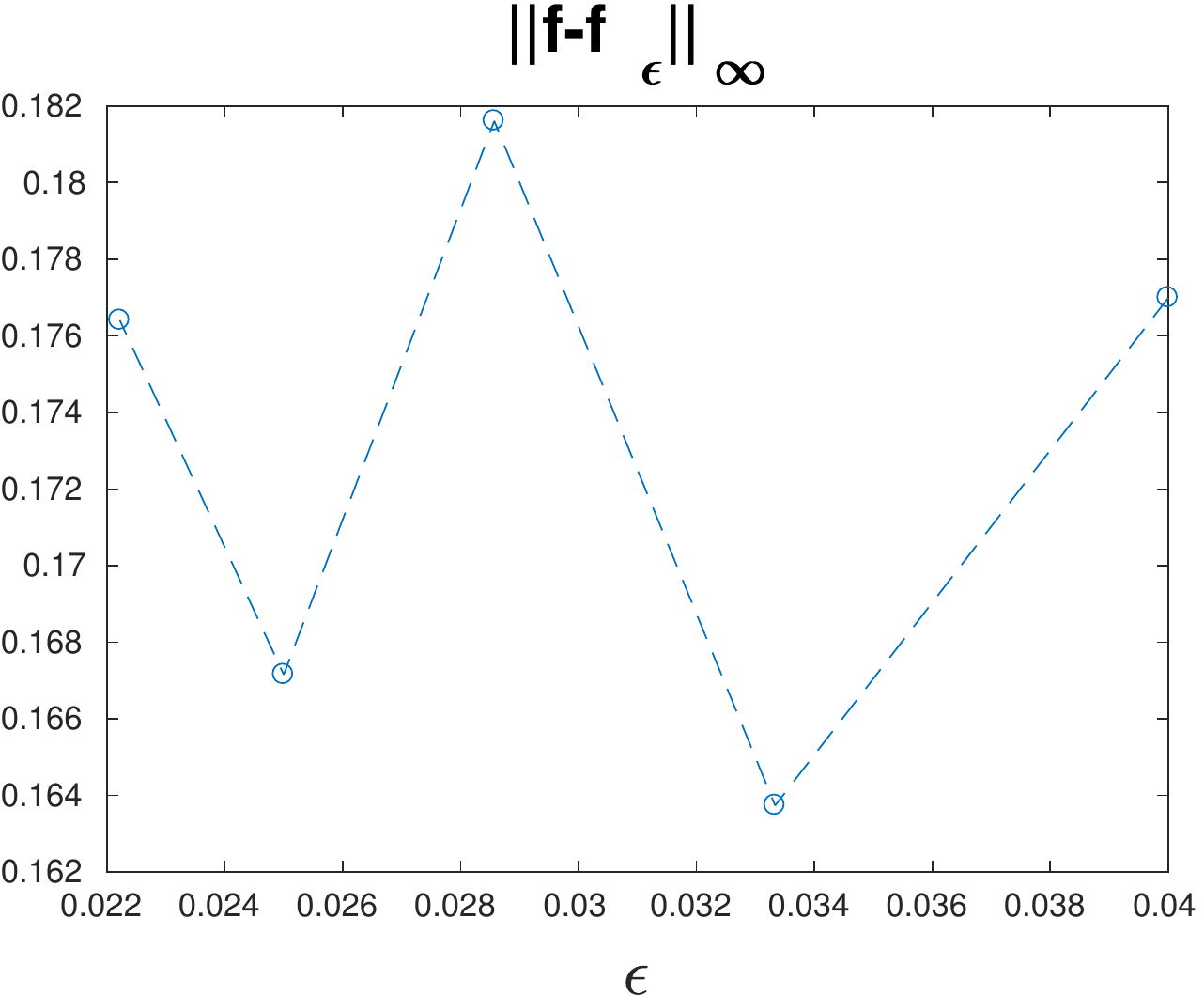} 
\caption{The three plots show the error $f_\Eps-f_0$ evaluated using different norms. The left panel shows the error in $L^2(\rd\eta\rd\theta)$ decreases to zero as $\varepsilon$ converges to zero, and the middle panel shows the convergence in $L^\infty(\rd\theta)$ at $\eta = 6$, which confirms the convergence of the end state. The plot on the right demonstrates the discrepancy of the error in $L^\infty(\rd\eta\rd\theta)$, and this is the plot that confirms the results in~\cite{WG2014}.}
\label{Figure:2}
\end{figure}

\begin{figure} 
\includegraphics[width = 0.5\textwidth]{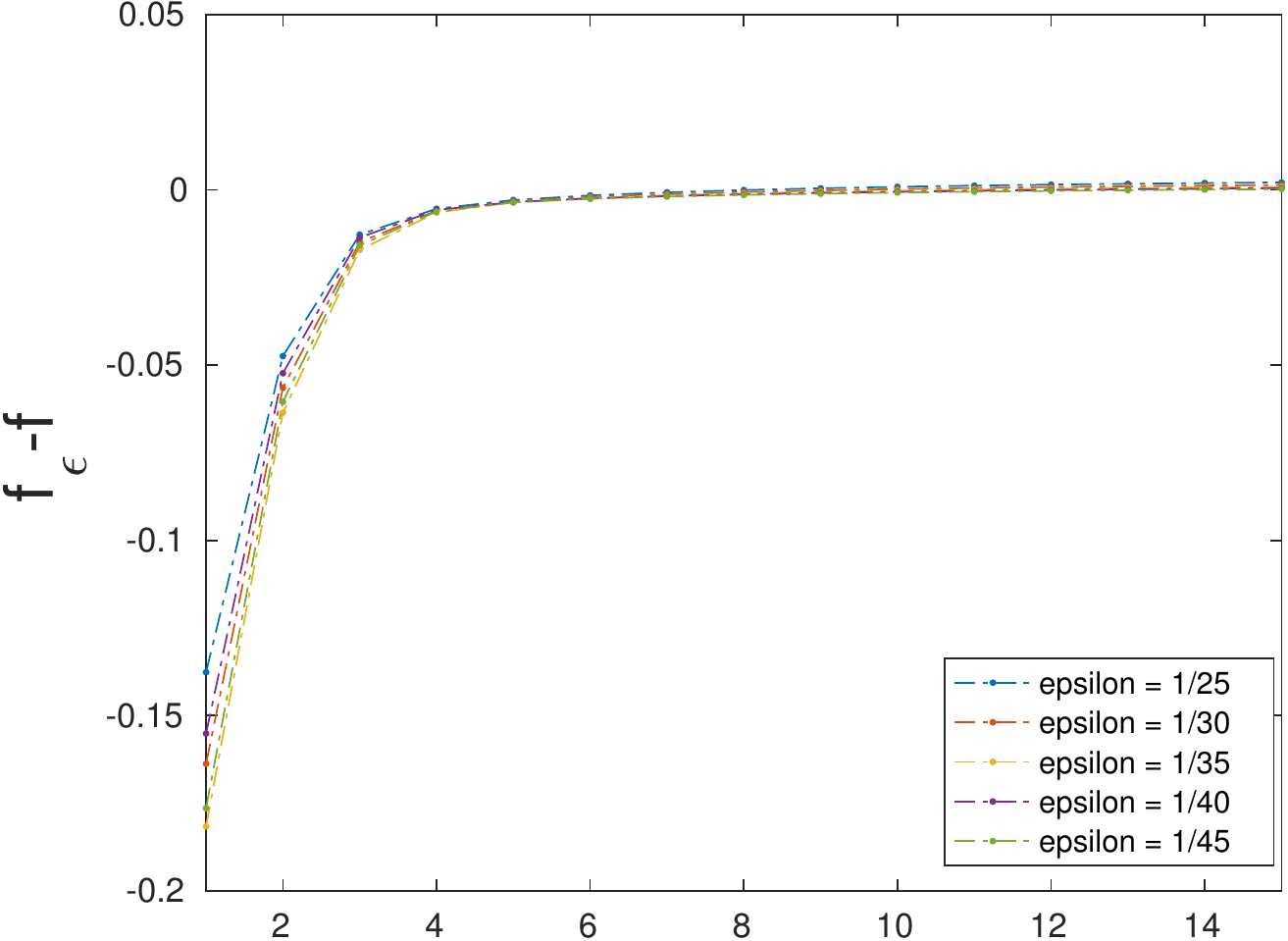}
\caption{Difference of the classical and the $\Eps$-Milne equation along the ray $(\eta, \theta) = (n \Eps, \Eps)$.}
\label{Figure:3}\
\end{figure}

\end{document}